%% file: DTGCE_Ver10.tex
\documentclass[10pt,twocolumn,twoside]{IEEEtran}

\usepackage{amsmath,amssymb,amsthm}
\usepackage{mathrsfs}
\usepackage{MnSymbol}

\usepackage{graphicx}
\usepackage{float}
\usepackage{caption}
\usepackage{subcaption}

\usepackage[linesnumbered,ruled]{algorithm2e}

\usepackage{cite}
\usepackage{enumitem}

\usepackage{tikz}
\usetikzlibrary{
    arrows,
    shapes,
    calc,
    intersections,
    through,
    backgrounds,
    positioning
}

\usepackage{pgfplots}
\pgfplotsset{compat=1.16}
\usepackage{pgfplotstable}

\usepackage{textgreek}
\usepackage{textcomp}

\usepackage{hyperref}
\hypersetup{
    colorlinks=true,
    linkcolor=blue,
    citecolor=blue,
    urlcolor=blue,
    filecolor=blue
}

\theoremstyle{plain}
\newtheorem{theorem}{Theorem}
\newtheorem{lemma}{Lemma}
\newtheorem{corollary}{Corollary}

\theoremstyle{definition}
\newtheorem{definition}{Definition}
\newtheorem{assumption}{Assumption}
\newtheorem{remark}{Remark}
\newtheorem{example}{Example}

\definecolor{myblue}{RGB}{0,114,178}
\definecolor{mygreen}{RGB}{0,158,115}

\def\BibTeX{{\rm B\kern-.05em{\sc i\kern-.025em b}\kern-.08em
    T\kern-.1667em\lower.7ex\hbox{E}\kern-.125emX}}

\begin{document}

\title{Output Feedback Guaranteed Cost Equilibrium in Networked Multi-Agent Linear-Quadratic Difference Games}
\author{Aniruddha Roy and Puduru Viswanadha Reddy, \IEEEmembership{Member, IEEE} \thanks{A. Roy is with the Robert Bosch Centre for Cyber-Physical Systems, Indian Institute of Science, Bengaluru, Karnataka, 560012, India. A. Roy's research work is supported by the Anusandhan National Research Foundation (ANRF) under the National Post-Doctoral Fellowship (NPDF) scheme, under a research grant PDF/2025/003882. P. V. Reddy is with the Department of Electrical Engineering, Indian Institute of Technology-Madras, Chennai, 600036, India.
		{(e-mail: aniruddharoy@iisc.ac.in, vishwa@ee.iitm.ac.in)}}}

\maketitle
\begin{abstract}
In this paper, we study infinite-horizon deterministic linear-quadratic difference games with an output feedback information structure. We consider linear time-invariant dynamics and quadratic cost functionals defined over an infinite horizon. We first demonstrate that computing an output feedback Nash equilibrium (OF-NE) in difference games is challenging, even for low-dimensional games. To address this difficulty, we introduce an output feedback guaranteed cost equilibrium (OF-GCE) for difference games. In an OF-GCE, each player seeks a feedback strategy that guarantees its cost remains below a prescribed bound while satisfying an equilibrium condition. We derive necessary and sufficient conditions for the existence of an OF-GCE in terms of the solvability of a set of coupled bilinear matrix inequalities. We provide a linear matrix inequality-based iterative algorithm for synthesizing OF-GCE strategies. We further show that the state feedback GCE (SF-GCE) is a special case of the OF-GCE when players have complete state information. Numerical examples illustrate the effectiveness of the proposed approach.
\end{abstract}
\section{Introduction} 
Game theory provides a mathematical framework for analyzing interactions among multiple decision-makers (also referred to as players or agents) \cite{Myerson:97}. Dynamic game theory (DGT) provides a modeling framework for analyzing interactions among players that evolve over time \cite{Basar:99}. In engineering applications, DGT has been widely used in areas such as cyber-physical systems \cite{Zhu:2015cps}, communication networks \cite{Zazo:16}, autonomous vehicles \cite{Fisac:2019:auto_veh_dyngame}, military and defense \cite{Isaacs:65}, and smart grids \cite{Saad:2012smartgrid}. In DGT, players' interactions are captured by a state vector which is governed by differential or difference equations, referred to as differential (continuous-time) and difference (discrete-time) games, respectively. In large-scale networked systems, agents depend on the local state (or output) information. Hence, it is impractical to use centralized control architectures that require complete state information. While most of the existing literature focuses on continuous-time games with complete state information, real-world systems typically operate in discrete time under local state information. These issues motivate us to develop an output feedback information-based control framework for discrete-time games.

This paper studies linear-quadratic (LQ) difference games with an output feedback information structure. We consider a class of difference games with linear time-invariant state dynamics and quadratic cost functionals defined over an infinite horizon. When a state feedback information structure is assumed, it is well known that the existence of a state feedback Nash equilibrium (SF-NE) for LQ difference games is related to the solvability of a set of coupled algebraic equations; see \cite[Theorem 3.2]{Monti:2024feedbackNash_discrete}. However, for the output information structure, the existence of an output feedback Nash equilibrium (OF-NE) for LQ difference games has received less attention. The work in \cite{Engwerda:08} studied the existence of OF-NE for LQ differential games, but solvability issues remain there due to the additional structural constraints; see \cite[Remarks 2.2 and 2.3]{Engwerda:08}. Building on the existing result for SF-NE in LQ difference games given in \cite[Theorem 3.2]{Monti:2024feedbackNash_discrete}, we extend the analysis to the case where players use output feedback information structure. In this setting, computing an OF-NE is also related to the solvability of a set of coupled algebraic equations. Unlike the state feedback case, the solution(s) of coupled algebraic equations must satisfy the additional structural constraint along with closed-loop system stability. We show that these conditions are often too restrictive to be met, even in low-dimensional games. To address this difficulty, we develop the notion of output feedback guaranteed cost equilibrium (OF-GCE). 

\subsubsection*{Contributions} 
This work aims to study an OF-GCE concept for infinite-horizon deterministic LQ difference games. We derive necessary and sufficient conditions for its existence. We also provide a computational method for synthesizing OF-GCE strategies. The main contributions of this work are summarized as follows: 
\begin{enumerate}
 \item We introduce the notion of an OF-GCE for LQ difference games. This concept is a discrete-time counterpart to our earlier work on feedback GCE for continuous-time LQ games; see \cite{Roy:22, Roy:2025guaranteed}. 
\item In Theorem \ref{thm:GCEtest_output_iff}, we provide necessary and sufficient conditions for a strategy profile to be a stabilizing OF-GCE. In particular, the conditions \eqref{eq:GCEoutput1}–\eqref{eq:GCEoutput2} can be directly used to check whether a given strategy profile is an OF-GCE. Moreover, to design each player’s feedback strategy given the strategies of the other players, Theorem \ref{thm:GCi_OF} provides necessary and sufficient conditions for the existence of such a feedback strategy.

\item We show that the computation of an OF-GCE is related to the feasibility of a non-convex set defined in \eqref{eq:setPi}. To address this non-convexity, we use semi-definite programming (SDP) relaxations to obtain convex approximations of the set \eqref{eq:setPi}. In addition, we provide a linear matrix inequality (LMI)-based iterative Algorithm \ref{alg:1_GCE_OF} for computing an OF-GCE. 
\item Our current approach unifies both the state and output feedback cases. In particular, we show that the state feedback guaranteed cost equilibrium (SF-GCE) arises as a special case of the OF-GCE framework. Corollary \ref{cor:FGi_SF} provides the synthesis of each player’s feedback strategy in the state feedback setting. Moreover, the associated feasibility set \eqref{eq:setPi_SF} is convex, and therefore no SDP relaxation is required for computation of an SF-GCE. 
 \end{enumerate}

\subsubsection*{Novelty and Differences with Existing Literature} 
The equilibrium concept studied in this paper is inspired by the satisfaction equilibrium for static games \cite{Ross:06}. For discrete-time settings, approximate SF-NE for LQ difference games were studied in \cite{Nortmann:2023approximateNash_discrete}, with data-driven computation methods developed in \cite{Nortmann:2024feedbackNash_discrete}. In \cite{Monti:2024feedbackNash_discrete}, the authors studied open-loop and feedback Nash equilibria in infinite-horizon LQ difference games. In contrast, we introduce a broader class of equilibrium strategies that guarantee individual costs remain below prescribed thresholds while retaining equilibrium properties. This work is a counterpart of our continuous-time GCE framework \cite{Roy:22, Roy:2025guaranteed} in discrete-time. While  \cite{Roy:22, Roy:2025guaranteed} provides only a sufficient condition for the existence of an OF-GCE in LQ differential games, here we derive the necessary and sufficient conditions for the existence of an OF-GCE. Additionally, we derive that an SF-GCE as a special case of OF-GCE. To the best of our knowledge, this is the first work on infinite-horizon LQ difference game under output feedback that ensures both equilibrium and cost guarantees.

The remainder of this paper is structured as follows. Section \ref{sec:preliminaries} presents preliminaries and problem formulation for discrete-time LQ games with output feedback information structure. Section \ref{sec:suboptimalLQ} provides fundamental results for the single-agent discrete-time suboptimal control problem. Building on these results, Section \ref{sec:GCE_concept_properties} revisits guaranteed cost responses and feedback guaranteed cost equilibria, and summarizes their key properties . Section \ref{sec:OF_GCE} presents existence and synthesis results for an OF-GCE. Section \ref{sec:SF_GCE} presents existence and synthesis results for an SF-GCE when players have complete state information. In particular, the SF-GCE results follow as a special case of the OF-GCE framework. Section \ref{sec:numerical_results} demonstrates our method through numerical examples. Finally, Section \ref{sec:concl} provides concluding remarks and future research directions.

\subsubsection*{Notation} 
$\mathbb{R}^n$ ($\mathbb{R}^n_+$) denotes the set of $n\times 1$ real (positive) column vectors, $\mathbb{R}^{n\times m}$ denotes the set of $n\times m$ real matrices. $E^\top$ denotes the transpose of a matrix or a vector $E$. ${I}_n$ denotes the $n\times n$ identity matrix. ${0}_{n\times m}$ denotes an $n\times m$ matrix with all its elements as zero. An $(n-1)$ tuple, $(1,2,\cdots,i-1,i+1,\cdots,n)$ is denoted by $-i$. An $n$-tuple, $(S_1,S_2,\cdots,S_n)$ is also denoted by $(S^i,S^{-i})$ where  $S^{-i} :=(S^1,\cdots,S^{i-1},S^{i+1},\cdots,S^n)$. A symmetric matrix $M\in\mathbb{R}^{n\times n}$ is positive definite (respectively, positive semi-definite) and denoted by $M\succ 0$ 
(respectively, $M\succeq 0$) if and only if $x_0^\top M x_0 > 0,~\forall\, x_0\in\mathbb{R}^n \setminus\{0\}$ (respectively, $x_0^\top M x_0 \ge 0,~\forall x_0\in\mathbb{R}^n$). The direct sum and Kronecker product of two matrices $A$ and $B$ are denoted by $A \oplus B$ and $A\bigotimes B$ respectively. $||A||$ denotes a matrix norm of the matrix $A$. $S\bigtimes T$ denotes the Cartesian product of the sets $S$ and $T$.   $\partial S$ denotes the boundary of a set $S$. A matrix $A \in \mathbb{R}^{n \times n}$ is called Schur stable if all its eigen values satisfy $\sigma(A) \subset \mathbb{D}$, where $\mathbb{D} := \{\lambda \in \mathbb{C} \mid |\lambda| < 1\}$. $A(i,j)$ denotes the element in the $i$-th row and $j$-th column of the matrix $A$. $\ker(A)$ denotes the kernel (null space) of $A$, i.e., $\ker(A) := \{ x \in \mathbb{R}^n \mid Ax = 0 \}$. Set $\mathsf{K} := \{0,1,2,\ldots\}$ denotes the set of non-negative integers.
\section{Preliminaries and problem formulation}
\label{sec:preliminaries}
In this section, we first review the preliminaries of an infinite-horizon nonzero-sum linear-quadratic difference game and subsequently present the problem formulation.  
\subsection{Preliminaries}
In this section, we consider an infinite-horizon nonzero-sum linear-quadratic difference game with partial state observations. The set of players (or agents) is denoted by $\mathsf{N}:=\{1,2,\ldots,N\}$. We assume that, at each time $k \in \mathsf{K}$, each player $i \in \mathsf{N}$, using the control $u_k^i\in \mathbb{R}^{m_{i}}$, influences the evolution of the state vector $x_k \in \mathbb{R}^n$ according to the following linear time-invariant dynamics in discrete time
\begin{subequations} 
\label{eq:LQDG} 
	\begin{align}
		& x_{k+1}= A x_k + \sum_{i \in \mathsf{N}} B^i u_k^i, 
        \label{eq:state_dynamics}
	\end{align}
with a given initial condition $x_0 \in \mathbb{R}^n$, where $A\in \mathbb R^{n\times n}$, $B^i\in \mathbb R^{n\times m_i}$. Player $i \in \mathsf{N}$ has partial state information, that is, access to the output vector $y_k^i\in \mathbb R^{s_i}$, given by 
	\begin{align} 
		& y_k^i = C^i x_k, ~ k \in \mathsf{K}, 
		\label{eq:output}
	\end{align}
    where $C^i \in \mathbb R^{s_i\times n}$ with $\text{rank}(C^i) = s_i$, and $s_i\leq n$. Each player $i$ seeks a control strategy $u_k^i$ that minimizes the quadratic cost functional
	\begin{align}
		J^i \left(u^i,u^{-i} \right) :=   \sum_{k=0}^{\infty} \left( {y_k^i}^\top Q^i y_k^i + {u_k^i}^\top R^i u_k^i \right),
\label{eq:objective_finitehorizon_output}
	\end{align}
	where $Q^i \in \mathbb R^{s_i\times s_i}$, $Q^i \succeq 0$, and $R^i \in \mathbb R^{m_i\times m_i}$, $R^i \succ 0$.
It is well established that the outcome of a noncooperative dynamic game depends on the information available to the players during the decision-making process, commonly referred to as the information structure; see \cite{Basar:99}. 
\end{subequations} 
\begin{assumption}
\label{asum:OFB_inforStruct}
The game \eqref{eq:LQDG} is played under an output feedback information structure, and each player $i \in \mathsf{N}$ uses a linear static output feedback control of the form
\begin{equation}
u_k^i = F^i y_k^i,~
F^i \in \mathbb{R}^{m_i \times s_i},~
k \in \mathsf{K}.
\end{equation}
\end{assumption}
Using Assumption \ref{asum:OFB_inforStruct}, the cost in \eqref{eq:objective_finitehorizon_output} rewritten as 
\begin{align}
		J^i \left(F^i ,F^{-i} \right)= \sum_{k=0}^{\infty} {x_k}^\top \left( {C^i}^\top Q^i C^i + {C^i}^\top {F^i}^\top R^i F^i C^i \right)x_k, \label{eq:playerfbcost}
	\end{align}
where $x_k$ evolves according to the closed-loop dynamics  $x_{k+1} = (A+\sum_{i\in \mathsf N} B^i F^i C^i)x_k,~k \in \mathsf{K}$, with a given initial condition $x_0 \in \mathbb{R}^n$. For notational simplicity, $J^i$ denotes both the original cost functional in \eqref{eq:objective_finitehorizon_output} and its reformulation  \eqref{eq:playerfbcost}.

\begin{assumption}
\label{asum:OFB_stabilizability}
The system \eqref{eq:state_dynamics} is output feedback stabilizable. That is, the following set of feedback strategies 
\begin{equation}
\mathscr{F}
:=
\Big\{
(F^i,F^{-i}) \in \bigtimes_{i\in\mathsf{N}} \mathbb{R}^{m_i\times s_i}
~\mid~
\sigma\Big(
A+\sum_{i\in\mathsf{N}} B^i F^i C^i
\Big)
\subset \mathbb{D}
\Big\},
\end{equation}
is non-empty. 
\end{assumption}
Due to the restriction imposed by Assumption \ref{asum:OFB_stabilizability}, the strategy spaces of the players become interdependent. However, this condition ensures Schur stability of the closed-loop system. As a result, the individual cost functionals are finite; see \cite{Engwerda:05}. The most widely studied solution concept in non-cooperative game theory is the Nash equilibrium. For non-cooperative LQ dynamic game \eqref{eq:LQDG} under an output feedback information structure in Assumption \ref{asum:OFB_inforStruct}, the output feedback Nash equilibrium is defined as follows.
%
\begin{definition}
An $N$-tuple $\left(F^{i\,\star},F^{-i\,\star}\right)\in \mathscr F$
is called an output feedback Nash equilibrium (OF-NE) if for each
player $i\in \mathsf N$, the following inequality holds:
\begin{align}
J^i\left(F^{i\,\star},F^{-i\,\star}\right)
\leq
J^i\left(F^i,F^{-i\,\star}\right),
\label{eq:NE}
\end{align}
for all $F^i$ such that
$\left(F^i,F^{-i\,\star}\right)\in \mathscr F$.
\end{definition}
Following studies in \cite{Nortmann:2023approximateNash_discrete, Nortmann:2024feedbackNash_discrete, Monti:2024feedbackNash_discrete}, which assume that each player $i \in \mathsf{N}$ uses a state feedback strategy $u_k^i = F^i x_k$, $k \in \mathsf{K}$, and we extend the analysis to an output feedback strategy $u_k^i = F^i y_k^i$, $k \in \mathsf{K}$ (refer to Assumption \ref{asum:OFB_inforStruct}). In this setting, the computation of an OF-NE for infinite-horizon LQ difference games is characterized by the joint solvability of the discrete-time coupled algebraic Riccati equations (DCARE) in \eqref{eq:DCARE},
\begin{subequations}
\begin{align}
\text{DCARE:}~& \Bigl(A + \sum_{j \in \mathsf{N}} B^j F^j C^j \Bigr)^\top P^i \Bigl(A + \sum_{j \in \mathsf{N}} B^j F^j C^j \Bigr)  - P^i \notag \\
&\qquad + {C^i}^\top Q^i C^i + {C^i}^\top {F^i}^\top R^i F^i C^i = 0, ~ i \in \mathsf{N}, \label{eq:DCARE}
\end{align}
and the equations in \eqref{eq:F_P_equation}, 
\begin{align}
M
\begin{bmatrix}
F^1 C^1 \\ F^2 C^2 \\ \vdots \\ F^N C^N
\end{bmatrix}
= -
\begin{bmatrix}
{B^1}^\top P^1 \\ {B^2}^\top P^2 \\ \vdots \\ {B^N}^\top P^N
\end{bmatrix} A,
\label{eq:F_P_equation}
\end{align}
\text{where}, 
\begin{align}
\resizebox{\columnwidth}{!}{$
M =
\begin{bmatrix}
R^1 + {B^1}^\top P^1 B^1 & {B^1}^\top P^1 B^2 & \cdots & {B^1}^\top P^1 B^N \\
{B^2}^\top P^2 B^1 & R^2 + {B^2}^\top P^2 B^2 & \cdots & {B^2}^\top P^2 B^N \\
\vdots & \vdots & \ddots & \vdots \\
{B^N}^\top P^N B^1 & {B^N}^\top P^N B^2 & \cdots &
R^N + {B^N}^\top P^N B^N
\end{bmatrix}.
$}
\label{eq:M_matrix}
\end{align}
\end{subequations}
From \eqref{eq:F_P_equation} and \eqref{eq:M_matrix}, we can write that  $ \bigl(R^i + {B^i}^\top P^i B^i \bigr) F^{i\,\star} C^i  + \sum_{j \neq i} {B^i}^\top P^i \left(B^jF^j C^j \right) = - {B^i}^\top P^i A$. From here, for each player $i$, we have 
\begin{align}
F^i C^i = -\bigl(R^i + {B^i}^\top P^i B^i \bigr)^{-1} {B^i}^\top P^i \bigl(A + \sum_{j \neq i} B^j F^j C^j),~i \in \mathsf{N}.  
\label{eq:OFNEstrategy}
\end{align} 
We define $A^{i} := \bigl(A + \sum_{j \neq i} B^j F^j C^j \bigr),~i \in \mathsf{N}$. Now, following \cite{Engwerda:08}, post-multiplying both sides of \eqref{eq:OFNEstrategy} by ${C^i}^\top$ gives $F^{i\,\star} C^i {C^i}^\top = - \bigl(R^i + {B^i}^\top P^i B^i\bigr)^{-1} {B^i}^\top P^i A^i {C^i}^\top$.  Since $\text{rank}(C^i) = s_i \leq n$, the matrix $C^i {C^i}^\top$ is invertible, which implies $F^{i\,\star} = - (R^i + {B^i}^\top P^i B^i )^{-1} {B^i}^\top P^i A^i {C^i}^\top (C^i {C^i}^\top )^{-1},~i \in \mathsf{N}$. For the exact solvability of $F^{i\,\star}$ in \eqref{eq:OFNEstrategy}, the solution from \eqref{eq:DCARE} and \eqref{eq:F_P_equation} must satisfy the following structural condition for each player $i \in \mathsf{N}$ 
\begin{align}
\label{eq:struc_condn}
\bigl(R^i + {B^i}^\top P^i B^i\bigr)^{-1} {B^i}^\top P^i A^i
\bigl(I_n - {C^i}^\top \bigl(C^i {C^i}^\top\bigr)^{-1} C^i\bigr) = 0.
\end{align}
We note that, since $R^i \succ 0$ and $P^i= {P^i}^\top \succeq 0$, the matrix $(R^i + {B^i}^\top P^i B^i)$ is invertible. Moreover, the matrix $I_{n} - {C^i}^\top (C^i{C^i}^\top )^{-1} C^i$ is the orthogonal projector onto $\ker(C^i)$. Therefore, the structural condition \eqref{eq:struc_condn}, implies that each row of ${B^i}^\top P^i A^i$, must belong to the row space of $C^i$ for $F^{i\,\star}$ to be exactly solvable.
\begin{remark}
\label{rem:DCARE_SFB_challenges}
Computing an SF-NE (when $C^i = I_n,~\forall i \in \mathsf{N}$) for LQ difference games is significantly more challenging than in the LQ differential games. In the LQ differential game setting, an SF-NE strategies are characterized by a set of $N$ numbers of coupled algebraic Riccati equations (CARE). The coupling with other players appears only through the CARE; see \cite[Equations 8.3.3 and 8.3.4]{Engwerda:05}. In this case, each player’s feedback gain depends only on their own solution of CARE. That is, once a set of solution $(P^i, P^{-i} )$ is computed, the corresponding strategies $(F^{i\,\star}, F^{-i\,\star})$ computed using $F^{i\,\star} = -({R^i})^{-1}{B^i}^\top P^i,~\forall i \in \mathsf{N}$; see \cite[Theorem 8.5]{Engwerda:05}.  In contrast, in the LQ difference games, an SF-NE is characterized  by a set of $2N$ coupled algebraic equations \eqref{eq:DCARE}-\eqref{eq:F_P_equation}.  In this case, the equilibrium strategy of each player directly depends on the strategies of all other players, as shown in \eqref{eq:OFNEstrategy}; see also \cite[Remark 1]{Nortmann:2023approximateNash_discrete} and  \cite[Remark 2]{Nortmann:2024feedbackNash_discrete}.  
\end{remark}
\begin{remark}
\label{rem:DCARE_variables}
In \eqref{eq:DCARE} and \eqref{eq:F_P_equation}, each player $i$ has two sets of unknowns: the feedback gain $F^i \in \mathbb{R}^{m_i \times s_i}$ and the symmetric positive definite matrix $P^i \in \mathbb{R}^{n \times n}$. Before any elimination, the total number of scalar decision variables is $\sum_{i=1}^{N} \left[ \frac{n(n+1)}{2} + m_i s_i \right]$. By eliminating $F^i$ using \eqref{eq:F_P_equation} and substituting it into  \eqref{eq:DCARE}, the problem reduces to a system of $N$ coupled algebraic equations involving only the unknown matrices $P^i$. The total number of scalar unknowns after this elimination becomes $N \cdot \frac{n(n+1)}{2}$ corresponding to all $P^i$ matrices. We further note that the resulting coupled equations are polynomial and, in general, non-quadratic in unknown variables. Unlike the Riccati equations encountered in continuous-time LQ games and in LQ optimal control (both continuous and discrete time), these equations are not quadratic in the decision variables; see also \cite[Remark 2]{Nortmann:2024feedbackNash_discrete}.  
\end{remark}
\begin{remark}
\label{rem:DCARE_OFB_challenges}
In addition to difficulties discussed in Remark \ref{rem:DCARE_SFB_challenges} and Remark \ref{rem:DCARE_variables}, computing an OF-NE is further complicated by the need to verification of the structural condition in \eqref{eq:struc_condn}. In the continuous-time case, this condition can be verified once the CARE solutions are obtained; see \cite[Equation 6]{Roy:2025guaranteed}. Nevertheless, computing an OF-NE  remains challenging even for LQ differential games as demonstrated in \cite{Roy:22, Roy:2025guaranteed}. In the discrete-time case, by contrast, the presence of the term $A^i$ in \eqref{eq:struc_condn} makes the verification of the structural condition more involved. As a result, computing an OF-NE becomes an even more difficult problem for LQ difference games. Note that, if players have complete state information (that is, $C^i = I_n$ for all $i \in \mathsf{N}$), then the structural condition \eqref{eq:struc_condn} is trivially satisfied, regardless of the solution of \eqref{eq:DCARE} and \eqref{eq:F_P_equation}. 
\end{remark}
The following example illustrates that the joint solvability of \eqref{eq:DCARE} and \eqref{eq:F_P_equation} may admit a solution. However, even after obtaining such a solution, verifying Schur stability of the closed-loop system and the structural condition in \eqref{eq:struc_condn} can be too restrictive. As a result, these conditions may fail to hold even for low-dimensional games.
\begin{example}
	\label{ex:OFNEex1}
    We consider a three-agent networked multi-agent system studied in \cite[Example~1]{Roy:22} and \cite[Example~2.3]{Roy:2025guaranteed}; see Fig.~\ref{fig:3agentnetwork1}. 
\begin{figure}[h]
    \centering
    \begin{tikzpicture}[scale=.25,>=latex', inner sep=1mm, font=\small]
        \tikzstyle{solid node}=[
            circle,
            auto=center,
            draw,
            minimum size=6pt,
            inner sep=2,
            fill=mygreen!20
        ]
        \tikzstyle{uedge}=[
            draw,
            black!65,
            -,
            line width=0.25mm
        ]

        \node (n1)[solid node] at (-7,0) {\scriptsize{$1$}};
        \node (n2)[solid node] at (0,0)  {\scriptsize{$2$}};
        \node (n3)[solid node] at (7,0)  {\scriptsize{$3$}};

        \path[uedge] (n1) -- (n2);
        \path[uedge] (n2) -- (n3);

    \end{tikzpicture}
    \caption{Undirected communication graph of the 3-agent multi-agent system in Example~\ref{ex:OFNEex1}.}
    \label{fig:3agentnetwork1}
\end{figure}
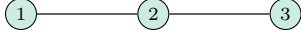
The continuous-time system in these references is discretized using a zero-order hold with a sampling time of $T_s=0.1$. The objective functions, output variables, and weighting matrices are identical to those in \cite{Roy:22, Roy:2025guaranteed}. For this reason, they are omitted for brevity. After eliminating the feedback gains $F^i$, $i=1,2,3$, the resulting DCARE in \eqref{eq:DCARE} consist of $18$ non-quadratic multivariate polynomial equations in $18$ variables. Moreover, imposing the structural condition \eqref{eq:struc_condn} requires the additional constraints $P^1(1,3)=0$ and $P^3(3,3)=0$. It can be verified that DCARE does not admit solutions that satisfy the structural constraint \eqref{eq:struc_condn}. As a result, the existence of an OF-NE for this discrete-time game is not clear. 
\end{example} 
The computation of an OF-NE for infinite-horizon LQ difference games is considerably more challenging than for LQ differential games \cite{Roy:22, Roy:2025guaranteed}; see also discussions in Remarks \ref{rem:DCARE_SFB_challenges}--\ref{rem:DCARE_OFB_challenges} and Example \ref{ex:OFNEex1}. We summarize the difficulties as follows: 
\begin{enumerate}
    \item In differential games, an OF-NE is characterized by a set of $N$ number of  CARE with variables $(P^i, P^{-i})$, and each player’s feedback gain depends only on their own solution. In contrast, for difference games it is characterized by a set of $2N$ coupled algebraic equations \eqref{eq:DCARE}-\eqref{eq:F_P_equation} which involves $(F^i, F^{-i}) $ and $(P^i, P^{-i})$ variables, where each player’s strategy explicitly depends on those of all other players; see Remark \ref{rem:DCARE_SFB_challenges}. 
    \item The coupled algebraic equations in LQ differential games are quadratic in the decision variables, whereas in LQ difference games they are non-quadratic; see Remark \ref{rem:DCARE_variables}.
   \item As in the differential game setting, once the coupled algebraic equations \eqref{eq:DCARE}-\eqref{eq:F_P_equation} are solved for difference games, we need to verify the Schur stability of the closed-loop systems. In other words, the stability requirement is not embedded while solving equations \eqref{eq:DCARE} and \eqref{eq:F_P_equation} and remains decoupled from the computation process; see also discussions in \cite{Engwerda:07, Roy:2025guaranteed}.

    \item Computing an OF-NE also requires verifying the structural condition \eqref{eq:struc_condn}. This verification is relatively straightforward in continuous time once the CARE solutions $(P^i, P^{-i})$ are obtained, but becomes significantly more involved in discrete time due to the presence of the additional term $A^i,~ i \in \mathsf{N}$; see Remark \ref{rem:DCARE_OFB_challenges}. 
\end{enumerate} 
\subsection{Problem statement}
The challenges outlined above for computing an OF-NE motivate the need to consider an alternative notion of equilibrium. To address this, we ask the following question. Can the recently studied guaranteed cost equilibria (GCE) concept, which form a larger class of strategies and include feedback Nash equilibrium when they exist, be extended from continuous-time to discrete-time LQ games in order to address the associated computational difficulties? To study an OF-GCE concept in discrete-time LQ games, we first analyze a single-player suboptimal control problem. Based on these results, we then develop an OF-GCE framework for discrete-time LQ games. In this formulation, instead of minimizing their costs, players seek output feedback strategies that guarantee their costs remain below prescribed thresholds.
\section{Suboptimal control for discrete-time linear systems}
\label{sec:suboptimalLQ}
In this section, we study the suboptimal linear-quadratic output feedback control problem for discrete-time linear systems. The results presented in this section provide a discrete-time analog of the continuous-time results presented in \cite[Section III]{Jiao:20TAC}. We consider the following discrete-time autonomous system
\begin{align}
    x_{k+1} = \bar{A} x_k, ~ k \in \mathsf{K},
    \label{eq:discrete_dynamics_1}
\end{align}
with a given initial condition $x[0] = x_0 \in \mathbb{R}^n$, where $\bar{A} \in \mathbb{R}^{n \times n}$. We assume the following quadratic performance index
\begin{align}
    J = \sum_{k=0}^{\infty} {x_k}^\top \bar{Q} x_k,
    \label{eq:performance_index_1}
\end{align}
where $\bar{Q} \in \mathbb{R}^{n \times n}$, $ \bar{Q} = \bar{Q}^\top \succeq 0 $. Our objective is to derive conditions under which the performance index in \eqref{eq:performance_index_1} for the system in \eqref{eq:discrete_dynamics_1} is upper bounded by a given threshold. To this end, we establish the following lemma.
\begin{lemma}
\label{lem:cost_discrete_lyap}
    Consider system \eqref{eq:discrete_dynamics_1} with the corresponding quadratic performance \eqref{eq:performance_index_1}. Assume $\bar{A}$ is Schur stable, that is, $\sigma(\bar{A}) \subset \mathbb{D}$. Then 
    \begin{enumerate}
    \item The cost \eqref{eq:performance_index_1} is finite, and it can be expressed as 
    \begin{align}
     J = x_0^\top Y x_0, 
    \end{align}
    where $ Y $ is the unique positive semi-definite solution to the Stein equation 
    \begin{align}
    \bar{A}^\top Y \bar{A} - Y + \bar{Q} = 0. 
    \end{align}
    %
    \item Alternatively, 
    \begin{align}
        J = \inf \left\{ x_0^\top P x_0 \ \middle| \ P \succ 0,~ \bar{A}^\top P \bar{A} - P + \bar{Q} \prec 0 \right\}.
    \end{align}
\end{enumerate}
\end{lemma}
\begin{proof}
The proof is a direct discrete-time counterpart of the continuous-time result in \cite[Lemma 3]{Jiao:20TAC}. Therefore, it is omitted. 
\end{proof}
In what follows, we present the necessary and sufficient conditions under which the cost $\eqref{eq:performance_index_1}$ satisfies $J < \delta$ for a given $\delta > 0$. This result is the discrete-time counterpart of the continuous-time result in \cite[Theorem 4]{Jiao:20TAC}, which is presented below. 
\begin{theorem}
\label{thm:subopt_DT_auto}
Consider system \eqref{eq:discrete_dynamics_1} with associated quadratic performance \eqref{eq:performance_index_1}.  
For a given $\delta > 0$, the following are equivalent
\begin{enumerate}
    \item $\sigma(\bar{A}) \subset \mathbb{D}$ and $J < \delta$. 
    \item There exists a matrix $P \succ 0$ such that
    \begin{subequations}
     \label{eq:LMI_P}
         \begin{align}
         & \bar{A}^\top P \bar{A} - P + \bar{Q} \prec 0,  \label{eq:LyaP_ineq_SOP}\\
         & x_0^\top P x_0 < \delta.  \label{eq:cost_bound_SOP}   
    \end{align}
    \end{subequations}
\end{enumerate}
\end{theorem}
\begin{proof}
The proof follows from Lemma \ref{lem:cost_discrete_lyap} and is omitted for brevity.
\end{proof}
Next, we consider the following discrete-time controlled linear time-invariant (LTI) system with output 
\begin{subequations}
\begin{align}
    x_{k+1} &= Ax_k + Bu_k, ~ k \in \mathsf{K}, 
    \label{eq:discrete_dynamics_2}  \\	
    y_k &= Cx_k,
\end{align}
\label{eq:controlledDynamics_OF_discrete}
\end{subequations}
with a given initial condition $x_0 \in \mathbb{R}^n$, where $A \in \mathbb{R}^{n \times n}, B \in \mathbb{R}^{n \times m}$. The quadratic cost functional is given by 
\begin{align}
    J(u) = \sum_{k=0}^\infty \left({y_k}^\top Q y_k + {u_k}^\top R u_k \right), 
    \label{eq:performance_index_2}
\end{align}
where $Q\in \mathbb R^{s\times s}$, $Q \succeq 0$, and $R\in \mathbb R^{m\times m}$, $R \succ 0$. 
We aim to obtain a static output feedback suboptimal control $u_k=Fy_k$ that ensures the associated cost
\begin{align}
	J(F):=  \sum_{k=0}^\infty {x_k}^\top \left(C^\top QC+C^\top F^\top RCF\right) x_k, \label{eq:Fcost}
\end{align}
is bounded above by a prescribed threshold $\delta > 0$, where $x_k$  evolves according to
\begin{align} 
	x_{k+1} = \left(A+BFC \right)x_k,~ k \in \mathsf{K}, \label{eq:linsystem}
\end{align}  
with a given initial condition $x_0 \in \mathbb{R}^n$. We assume that the system \eqref{eq:controlledDynamics_OF_discrete} is output feedback stabilizable, that is, there exists a gain $F \in \mathbb{R}^{m \times s}$ such that $\sigma(A+BFC) \subset \mathbb{D}$. The following result provides a necessary and sufficient condition for the existence of such a suboptimal static output feedback control, which is a direct consequence of Theorem \ref{thm:subopt_DT_auto}.  
\begin{theorem}
\label{thm:GCC_discrete_iff}
Consider the  system \eqref{eq:controlledDynamics_OF_discrete} with the associated cost functional \eqref{eq:performance_index_2}. Assume the control law $u_k = Fy_k$. Let $ \delta > 0 $ be a given upper bound. Then, the following statements are equivalent 
\begin{enumerate}
    \item $\sigma(A+BFC) \subset \mathbb{D} $ and $ J(F) < \delta$. 
    \item There exist $P \succ 0$  and $F \in \mathbb{R}^{m \times s}$such that
 \begin{subequations}
\begin{align}
&\resizebox{0.9\columnwidth}{!}{$
\left(A+BFC \right)^\top P \left(A+BFC \right)-P+C^\top QC+C^\top F^\top RFC \prec 0,
$}
\label{eq:lyaP_ineq_discrete}\\
&x_0^\top P x_0 < \delta.
\label{eq:cost_bound}
\end{align}
\end{subequations}
\end{enumerate}
\end{theorem}
\begin{proof}
Using the control law $u_k= FC x_k$, the dynamics in \eqref{eq:discrete_dynamics_2} and the cost functional in \eqref{eq:performance_index_2} become to $x_{k+1} = \bar{A} x_k,~J(F) = \sum_{k=0}^{\infty} x_k^\top \bar{Q} x_k$, where $\bar{A} := A + B F C$ and $\bar{Q} := C^\top Q C + C^\top F^\top R F C$. Hence, the proof follows the same lines as that of Theorem \ref{thm:subopt_DT_auto}.
\end{proof}
\section{Guaranteed cost equilibrium in infinite horizon difference games}
\label{sec:GCE_concept_properties}
Based on the concept of suboptimal control, the notion of a guaranteed cost function (GCF) was introduced in our earlier work on differential games; see \cite{Roy:22, Roy:2025guaranteed, Roy:24}. This GCF concept led to the corresponding equilibrium notion. In the present work on difference games, the definition of the GCF, the corresponding equilibrium, and their properties remain unchanged. Therefore, we briefly review these concepts for clarity and completeness. Let the $N$-tuple $(\delta^i, \delta^{-i}) \in \mathbb{R}_+^N$ represent a cost profile, which is a design parameter. We recall the set-valued GCF of each player $i \in \mathsf{N}$ as follows
\begin{align}
\resizebox{0.98\columnwidth}{!}{$
\mathfrak{g}^i(\delta^i,F^{-i})
:=
\left\{
F^i\in\mathbb{R}^{m_i\times s_i}
~\middle|~
J^i(F^i,F^{-i})<\delta^i,\,
(F^i,F^{-i})\in\mathscr{F}
\right\}.
$}
\label{eq:satisfactionfunction}
\end{align} 
Here, $\mathfrak g^i(\delta^i,F^{-i})$ denotes the set of feedback strategies that guarantee player $i$'s cost does not exceed the specified bound $\delta_i$, given that the remaining players use their feedback strategies $F^{-i}$. Inspired by this idea, we recall the equilibrium concept in the following definition, as introduced in \cite{Roy:22, Roy:2025guaranteed}. 
\begin{definition}
\label{def:FGCE}
Let $ \left(\delta^i,\delta^{-i} \right)\in \mathbb R^N_+$ be a given cost profile. 
	The strategy profile $  F^\circ: =\left(F^{i \, \circ}, F^{-i \, \circ} \right){~\in \mathscr F}$ is an \emph{output feedback guaranteed cost equilibrium} (OF-GCE) if for each player $i \in \mathsf{N}$ the following condition holds true 
	\begin{subequations}
		\begin{align}
			F^{i \, \circ}\in \mathfrak g^i (\delta^i, F^{-i \, \circ} ). \label{eq:fixedpoint}
		\end{align} 
		The set of all OF-GCE is given by
		\begin{align}
			\mathscr{F}^\circ:=\{ (F^i,F^{-i} )\in {\mathscr F} ~|~F^i\in \mathfrak g^i (\delta^i,F^{-i \, \circ}),~ \forall i\in \mathsf N \}.
		\end{align}
	\end{subequations}
\end{definition}
\begin{remark}
As discussed in our earlier work \cite{Roy:22, Roy:2025guaranteed}, we note that once the players reach an OF-GCE, no player has an incentive to deviate from their current strategy, as each player has achieved the desired upper bound on their cost. In other words, players are not interested in changing their strategies once they are satisfied. Additionally, an interpretation of OF-GCE strategies, and their conceptual similarity to correlated equilibria in static games \cite{Myerson:97}, is discussed in \cite[Remark 4.3]{Roy:2025guaranteed}. 
\end{remark}
In what follows, we recall the important properties of OF-GCE from our earlier work in \cite{Roy:22, Roy:2025guaranteed}. 
\begin{theorem} 
Let $\left(\delta^i,\delta^{-i} \right)\in \mathbb R^N_+ $ and $ \left(\bar{\delta}^i,\bar{\delta}^{-i} \right)\in \mathbb R^N_+ $ be two cost profiles such that $\delta^i\leq \bar{\delta}^i$ for all $i\in \mathsf N$,
	and let $\mathscr F^\circ $ and $\bar{\mathscr  F}^\circ $ denote the associated
	set of  OF-GCE strategies respectively. Then, $\mathscr  F^\circ \subseteq \bar{\mathscr  F}^\circ$. 
	\label{thm:LargeDelta}
\end{theorem}
\begin{theorem}
    Let $\left(F^{i\,\star}, F^{-i\,\star} \right)$ be an OF-NE. Then, $\left(F^{i\,\star}, F^{-i\,\star} \right)$ is also an OF-GCE considering the cost profile $\delta^i=J^i(F^{i\,\star},F^{-i\,\star})+\xi$,~$i\in \mathsf N$, for some $\xi > 0$.
    \label{thm:GCE_Nash}
\end{theorem}
Theorem~\ref{thm:LargeDelta} highlights a monotonicity property of OF-GCE. As a consequence of Theorem~\ref{thm:LargeDelta}, Theorem~\ref{thm:GCE_Nash} shows that if an OF-NE exists for the difference game \eqref{eq:LQDG}, then, for an appropriate choice of the cost profile, this OF-NE is also an OF-GCE. As a result, Theorem~\ref{thm:GCE_Nash} implies that the set of guaranteed cost equilibrium strategies is broader than the set of Nash equilibrium strategies. Illustrations of the properties presented in Theorems~\ref{thm:LargeDelta} and \ref{thm:GCE_Nash} can be found in \cite[Figure~2(a)]{Roy:2025guaranteed} and \cite[Figure~2(b)]{Roy:2025guaranteed}, respectively.

Equilibrium outcomes are typically inefficient when compare to a cooperative outcome, since players can jointly reduce their costs through cooperation. In both discrete-time \cite{Lin2020:pareto_infiniteHorizon} and continuous-time \cite{Reddy:13} settings, all Pareto-optimal solutions of the cooperative game \eqref{eq:LQDG} can be obtained via weighted-sum optimal control with weights in the 
$N$-dimensional unit simplex. The Price of Stability (PoS) is a metric commonly used to capture the efficiency loss due to the stability requirement imposed by equilibrium behavior; see \cite{ Anshelevich:08}. The PoS associated with an OF-GCE  strategy profile $F^\circ = (F^{i\,\circ}, F^{-i\,\circ})\in \mathscr{F}^\circ$ defined as  
\begin{align}
	\mathrm{PoS}(F^\circ) :=\frac{\sum_{i \in \mathsf{N}} J^i \left(F^{i\,\circ}, F^{-i\,\circ} \right)}{J_\mathrm{Co}},
	\label{eq:PoSdef}
\end{align}
where $J_\mathrm{Co}$ is the total game cost in cooperation, which is obtained by solving the following optimal control problem 
\begin{align} 
	J_\mathrm{Co}:=\min_{(F^i,F^{-i})}  \sum_{i \in \mathsf{N}} J^{i} \left(F^i,F^{-i} \right), 
	\label{eq:tgcostcoop}
\end{align} 
subject to $x_{k+1} =\left(A+\sum_{i\in \mathsf N} B^i F^i C^i\right)x_k$, $k \in \mathsf{K}$, with a given initial condition $x_0 \in \mathbb{R}^n$. We denote  $\bar{Q}=\sum_{i\in \mathsf N} {C^i}^\top Q^i C^i$, $\bar{R}=\oplus_{i\in \mathsf N}R^i$, and $\bar{B}=[B^1~B^2~\cdots~B^N]$. Since $\bar{Q}\succeq 0$ and $\bar{R}\succ 0$, the optimal cooperative cost \eqref{eq:tgcostcoop} is given by $x_0^\top P x_0$ (see \cite[Theorem 8]{Lin2020:pareto_infiniteHorizon}), 
where $P$ denotes the unique positive definite stabilizing solution of the following 
discrete-time algebraic Riccati equation (DARE).
\begin{align}
	\mathrm{DARE}:~	A^\top PA - P + \bar{Q}-A^\top P \bar{B}\left(\bar{R}+\bar{B}^\top P \bar{B} \right)^{-1} \bar{B}^\top PA =0. \label{eq:DARE}
\end{align}
\begin{theorem}
\label{thm:PoS}
	Let $\left(\delta^i,\delta^{-i} \right)\in \mathbb R^N_+ $ be a  cost profile and let $\mathscr{F}^\circ$ be the set of OF-GCE associated with the difference game \eqref{eq:LQDG}. Let $P$ be a unique positive definite stabilizing solution of DARE \eqref{eq:DARE}. Then, the PoS associated with any $F^\circ\in \mathscr F^\circ$ satisfies 
	\begin{align}
		1\leq	\mathrm{PoS}(F^\circ) < \frac{\sum_{i \in \mathsf{N}} \delta^i }{x_0^\top P x_0}.
		\label{eq:PoS1}
	\end{align}
\end{theorem}
\begin{remark}
\label{rem:posbound}
    From \eqref{eq:PoS1}, the PoS bounds are uniform, because they hold for all $(F^{i \,\circ}, F^{-i\,\circ}) \in \mathsf{F}^\circ$. The upper bound implies that OF-GCE solutions lie between the Pareto frontier and the hyperplane $\sum_{i \in \mathsf{N}} J^i = \sum_{i \in \mathsf{N}} \delta^i$. The feasible cost profiles satisfy $\{(\delta^i,\delta^{-i}) \in \mathbb{R}^N_+ \mid \sum_{i \in \mathsf{N}} \delta^i > x_0^\top P x_0\}$, that is, the sum of guaranteed cost profiles of all players must exceed the total cooperative cost. The feasible region tightens as $\sum_i \delta^i$ decreases, which is consistent with Theorem \ref{thm:LargeDelta}. Note that the cost profile $(\delta^i,\delta^{-i})$ is a parameter associated with an OF-GCE and cannot be selected individually by the players; see \cite[Figures 2(c) and 2(d)]{Roy:2025guaranteed} for an illustration. We also note that, if more than one OF-GCE exists, the natural choice is the one with the lowest PoS, since it comes closest to achieving the welfare of the cooperative outcome. 
\end{remark}
\begin{remark} 
Note that we focus on the non-weighted optimization problem in \eqref{eq:tgcostcoop}. 
If a weighted formulation $\min \sum_{i \in \mathsf{N}} \alpha^i J^i(F^i,F^{-i})$, with $\alpha^i>0$ and $\sum_i \alpha^i=1$, is considered, then  $ \min \sum_{i \in \mathsf{N}} \alpha^i J^i(F^i, F^{-i}) \leq \sum_{i \in \mathsf{N}} \alpha^i J^i(F^{i\,\circ}, F^{-i\,\circ}) < \sum_{i \in \mathsf{N}} \alpha^i \delta^i$. This shows that, at an OF-GCE, the joint weighted cost lies between the Pareto frontier and the hyperplane $\sum_i \alpha^i \delta^i$. Since equilibrium efficiency is often quantified using standard metrics such as the Price of Anarchy (PoA) and the PoS \cite{Anshelevich:08}, we restrict attention to the non-weighted sum of players’ costs when computing the PoS for an OF-GCE.
\end{remark}
\begin{remark}
Following \cite[Theorem 2.1]{Engwerda:08} and its discrete-time counterpart, the output feedback cooperative strategies $(F^{i}_{\mathrm{Co}}, F^{-i}_{\mathrm{Co}})$ that achieve the optimal cooperative cost must satisfy $\left[\begin{smallmatrix}
     (F^{1}_{\mathrm{Co}} C^1 )^\top & (F^{2}_{\mathrm{Co}} C^2 )^\top & \ldots & (F^{N}_{\mathrm{Co}} C^N )^\top
\end{smallmatrix} \right]^\top 
= {\small{-(\bar{R} + \bar{B}^\top P \bar{B} )^{-1} \bar{B}^\top P A}}.$ Consequently, the implementable output feedback cooperative strategy $F^{i}_{\mathrm{Co}}$ for player $i \in \mathsf{N}$ is obtained by solving $F^{i}_{\mathrm{Co}}C^i = -E^i (\bar{R} + \bar{B}^\top P\bar{B} )^{-1}\bar{B}^\top PA,~i \in \mathsf{N},$ where $E^i = \begin{bmatrix} 0_{m_i \times m_1} & \ldots & I_{m_i} & \ldots & 0_{m_i \times m_N} \end{bmatrix}$. From this, cooperative output feedback strategies are uniquely solvable if the stabilizing solution $P \succ 0$ of \eqref{eq:DARE} satisfies $E^i (\bar{R} + \bar{B}^\top P\bar{B})^{-1}\bar{B}^\top PA(I - {C^i}^\top(C^i {C^i}^\top)^{-1}C^i) = 0$, $~\forall i \in \mathsf{N}$. As in this case, $x_0^\top P x_0$ provides a lower bound on the cooperative cost.
It is well studied that in distributed control of multi-agent systems, output feedback strategies achieving this bound may not exist \cite{Jiao:20TAC}. However, our approach ensures that the upper bound in \eqref{eq:PoS1} is derived by comparing the total cost at an OF-GCE with this theoretical lower bound. 
\end{remark}
\section{Existence and synthesis of Output feedback guaranteed cost equilibria}
\label{sec:OF_GCE}
In this section, we derive necessary and sufficient conditions for a strategy profile to be a stabilizing OF-GCE. We also provide an iterative algorithm to synthesizing them. 
\subsection{Output feedback guaranteed cost equilibrium}
We begin with a necessary and sufficient conditions for a strategy profile to be a stabilizing OF-GCE.
\begin{theorem}
\label{thm:GCEtest_output_iff}
Let a cost profile $ \left(\delta^i ,\delta^{-i} \right) \in \mathbb{R}^N_+ $ and a strategy profile $ F^\circ := (F^{i\,\circ}, F^{-i\,\circ})$ be given. Define the closed-loop system as $A_\text{cl}(F^\circ) := A + \sum_{j\in \mathsf N} B_j F_{j}^{\circ} C_j.$ Then, for each player $ i \in \mathsf N $, the following statements are equivalent 
\begin{enumerate}
    \item The strategy profile $F^\circ $ is a stabilizing OF-GCE, that is, $ J^i(F^{i\,\circ}, F^{-i\,\circ}) < \delta^i$ and $\sigma(A_\text{cl}(F^\circ)) \subset \mathbb{D}$. 
   \item There exists a matrix $ P^i \succ 0 $ such that
\begin{subequations} %
\label{eq:GCEtest_outputfeedback} 
\begin{align}
&\resizebox{0.90\columnwidth}{!}{$
A_{\mathrm{cl}}(F^\circ)^\top P^i A_{\mathrm{cl}}(F^\circ)
- P^i
+ {C^i}^\top Q^i C^i
+ {C^i}^\top {F^{i\,\circ}}^\top R^i F^{i\,\circ} C^i
\prec 0,
$}
\label{eq:GCEoutput1}\\
& x_0^\top P^i x_0 < \delta^i.
\label{eq:GCEoutput2}
\end{align}
\end{subequations}
\end{enumerate}
\end{theorem}
\begin{proof}
 The proof is similar to that of Theorem \ref{thm:GCC_discrete_iff}, with $A + BFC$, $C^\top QC + C^\top F^\top RFC$, and $P$ replaced by $A_\text{cl}(F^\circ)$, ${C^i}^\top Q^i C^i+ {C^i}^\top {F^{i\, \circ}}^\top R^i F^{i\,\circ} C^i$, and $P^i$, respectively, for each player $i \in \mathsf{N}$. 
\end{proof}
Theorem \ref{thm:GCEtest_output_iff} provides a necessary and sufficient conditions for a strategy profile to be a stabilizing OF-GCE. Note that, the conditions \eqref{eq:GCEoutput1}-\eqref{eq:GCEoutput2} can directly be used with the variable $P^i \succ 0,~i \in \mathsf{N}$, to check whether a given strategy profile is an OF-GCE. However, Theorem \ref{thm:GCEtest_output_iff} does not provide a way to synthesize an OF-GCE, which requires the non-emptiness of the guaranteed cost response. The next result provides the required necessary and sufficient conditions. 
\begin{lemma}
\label{lem:GCi_output_iff}
Let $\delta^i>0$ and $F^{-i}\in \bigtimes_{j \in -i} \mathbb R^{m_j \times s_{j}}$ be given, and define $ A^i := A + \sum_{j \in -i} B^j F^j C^j$. Then, for each player $ i \in \mathsf N $, the following statements are equivalent
\begin{enumerate}
    \item The strategy $F^i \in \mathbb{R}^{m_{i} \times s_{j}}$ belongs to  the set $\mathfrak g^i(\delta^i, F^{-i})$, that is, $\mathfrak g^i(\delta^i,F^{-i}) \neq \emptyset$ and $\sigma(A+\sum_{i\in \mathsf N} B^iF^i C^i) \subset \mathbb{D}$. 
    \item There exists a pair $(P^i, F^i)$ with $P^i\succ 0$ and $F^i\in \mathbb R^{m_i\times s_i}$ such that the following conditions hold
  \begin{subequations}
\begin{align}
& \bigl(A^i + B^i F^i C^i\bigr)^\top P^i
\bigl(A^i + B^i F^i C^i\bigr)
- P^i + {C^i}^\top Q^i C^i \nonumber\\
& \qquad + {C^i}^\top {F^i}^\top R^i F^i C^i
\prec 0,
\label{eq:LMI1_OF}\\
& x_0^\top P^i x_0 < \delta^i.
\label{eq:LMI2_OF}
\end{align}
  \end{subequations}  
\end{enumerate}
\end{lemma}
\begin{proof}
The proof follows along the same lines of the proof of Theorem \ref{thm:GCEtest_output_iff}.
\end{proof}
Note that the inequality \eqref{eq:LMI1_OF} is a bilinear matrix inequality (BMI) in the variables $P^i$ and $F^i$, whose feasibility is generally nonconvex problem \cite{Mesbahi:2000}. Nevertheless, we show that the feasibility of \eqref{eq:LMI1_OF} together with \eqref{eq:LMI2_OF}, and hence the non-emptiness of the guaranteed cost response, can be verified using the projection lemma \cite{Iwasaki:94, Skelton:17}, which we recall in the following result.
\begin{lemma}
	Let $  U \in \mathbb R^{n\times m}$, $\text{rank}(U)=m<n$,
	$V\in \mathbb R^{s\times n}$, $\text{rank}(V)=s<n$, and 
	$\Phi \in \mathbb R^{n\times n}$,~$\Phi=\Phi^\top$, be given. Then, there exists $F\in \mathbb R^{m\times s}$
	satisfying $\Phi+U F V+ (U F V)^\top  \prec 0$
	if and only if $\mathcal N_{U^\top}^\top \Phi \mathcal N_{U^\top} \prec 0$ and $\mathcal N_V^\top \Phi \mathcal N_V \prec 0$ hold, where $\mathcal N_{S^\top}
	\in \mathbb R^{n\times (n-m)}$, $\mathcal N_V\in \mathbb R^{n\times (n-s)}$, denoting any matrices whose columns form orthonormal
	bases of the null spaces of $U^\top$, $V$ respectively.
	\label{lem:Finsler}
\end{lemma} 
\begin{theorem}
\label{thm:GCi_OF}
Let $\delta^i>0$ and $F^{-i} \in \bigtimes_{j \in -i} \mathbb R^{m_j \times s_j} $ be given.    Consider the following sets
\begin{subequations}
\begin{align}
\mathscr{X}^i
&:= \Big\{X \in \mathbb{R}^{n\times n}~\big|~
X \succ 0,\;
\mathcal{N}_{\tilde{C}^i}^\top
\Phi_1^i(X)
\mathcal{N}_{\tilde{C}^i}
\prec 0
\Big\},
\label{eq:XLMI}\\
\mathscr{Y}^i
&:= \Big\{Y \in \mathbb{R}^{n\times n}~\big|~
Y \succ 0,\;
\begin{bmatrix}
\delta^i & x_0^\top\\
x_0 & Y
\end{bmatrix}
\succ 0,
\nonumber\\
&\qquad \qquad \qquad \qquad \quad \quad
\mathcal{N}_{\tilde{B}^i}^\top
\Phi_2^i(Y)
\mathcal{N}_{\tilde{B}^i}
\prec 0
\Big\}, 
\label{eq:YLMI}
\end{align}
\end{subequations}
\text{where,} 
\begin{subequations}
    \begin{align}
            \Phi_1^i \left(X\right) &:=\begin{bmatrix} 
				-X  &{A^i}^\top X &\left(\sqrt{Q^i}C^i \right)^\top  & 0_{n\times  {m}_i }\\
                X A^i  &-X &0_{n \times s_{i}} &0_{n \times m_{i}} \\
				\sqrt{Q^i}C^i   &0_{s_{i} \times n} & -I_{s_i} &0_{ s_i \times m_i}\\
				0_{{m}_i\times n} &0_{{m}_i\times n} &0_{m_i\times s_i }&-(R^i)^{-1} \end{bmatrix},\\
			\Phi_2^i \left(Y \right)&:=\begin{bmatrix} 
                -Y  &Y {A^i}^\top &\left (\sqrt{Q^i}C^i Y \right)^\top  & 0_{n\times  {m}_i }\\
                A^iY  &-Y &0_{n \times s_{i}} &0_{n \times m_{i}} \\
				\sqrt{Q^i}C^iY   &0_{s_{i} \times n} & -I_{s_i} &0_{ s_i \times m_i}\\
				0_{{m}_i\times n} &0_{{m}_i\times n} &0_{m_i\times s_i }&-(R^i)^{-1}
				 \end{bmatrix},
        \end{align}
        \end{subequations}
		and $A^i:=A+\sum_{j\in -i}B^j F^jC^j$. The matrices $\mathcal N_{\tilde{C}^i}:= \ker(\tilde{C}^i)$ and $\mathcal N_{\tilde{B}^i}:= \ker(\tilde{B}^i)$ denote matrices with orthonormal columns which span the null spaces of the matrices $\tilde{C}^i := \begin{bmatrix} C^{i} &0_{s_{i} \times n} &0_{s_{i} \times s_{i}} &0_{s_{i} \times m_{i}} \end{bmatrix}$ and $\tilde{B}^i := \begin{bmatrix} 0_{m_{i} \times n} &{B^i}^\top & 0_{m_{i} \times s_{i}} &I_{m_i}\end{bmatrix}$ respectively. Define the set
		\begin{align}
			\mathscr{P}^i&:=\left \{P^i\in \mathbb R^{n\times n} ~|~P^i\succ 0,  ~P^i \in \mathscr{X}^i, ~(P^i)^{-1} \in \mathscr{Y}^i \right\}.\label{eq:setPi}
		\end{align}
		Then, there exists an $ F^i \in \mathfrak{g}^i(\delta^i, F^{-i}) \neq \emptyset$, and  for any $F^i \in \mathfrak{g}^i(\delta^i, F^{-i}) $ we have $\sigma(A^i + B^i F^i C^i) \subset \mathbb{D}$ if and only if $ \mathscr{P}^i \neq \emptyset $. In this case, any $ F^i \in \mathfrak{g}^i(\delta^i, F^{-i}) $ satisfies the following LMI for some $ P^i \in \mathscr{P}^i $ 
		\begin{align}
\resizebox{0.98\columnwidth}{!}{$
\begin{bmatrix}
-P^i
&
\bigl(A^i + B^i F^i C^i\bigr)^\top P^i
&
\bigl(\sqrt{Q^i}C^i\bigr)^\top
&
\bigl(\sqrt{R^i}F^iC^i\bigr)^\top
\\
P^i\bigl(A^i + B^i F^i C^i\bigr)
&
-P^i
&
0_{n\times s_i}
&
0_{n\times m_i}
\\
\sqrt{Q^i}C^i
&
0_{s_i\times n}
&
-I_{s_i}
&
0_{s_i\times m_i}
\\
\sqrt{R^i}F^iC^i
&
0_{m_i\times n}
&
0_{m_i\times s_i}
&
-I_{m_i}
\end{bmatrix}
\prec 0.
$}
\label{eq:FGCi}
\end{align} 
\end{theorem}
\begin{proof}
\textbf{$\Rightarrow$} First, we show that $\mathscr{P}^i \neq \emptyset$ if there exists  $F^i \in \mathfrak{g}^i(\delta^i, F^{-i})$ such that $A^i + B^i F^i C^i$ is Schur stable and $J^i(F^i, F^{-i}) < \delta^i$. To this end, we assume that such an $F^i$ exists. That is $F^i \in \mathfrak{g}^i(\delta^i, F^{-i})$ ensures the Schur stability of $A^i + B^i F^i C^i$ and satisfies $J^i(F^i, F^{-i}) < \delta^i$.
 This implies that the conditions \eqref{eq:LMI1_OF} and \eqref{eq:LMI2_OF} in Lemma \ref{lem:GCi_output_iff} are satisfied. Rewrite BMI \eqref{eq:LMI1_OF} as follows
\begin{align}
&-P^i + {C^i}^\top Q^i C^i
+ {C^i}^\top {F^i}^\top R^i F^i C^i
\nonumber\\
&\qquad
+ \bigl(A^i + B^i F^i C^i\bigr)^\top
P^i \bigl({P^i} \bigr)^{-1} P^i
\bigl(A^i + B^i F^i C^i\bigr)
\prec 0.
\label{eq:BMI_1}
\end{align}
Since $Q^i\succeq 0$, we write ${Q}^i=\sqrt{Q^i} (\sqrt{Q^i})^\top$. Using the Schur complement and the  notation provided in the theorem statement, the BMI  \eqref{eq:BMI_1} is written as 
\begin{align}
    \Phi_1^i(P^i) + (\bar{B}^i)^\top F^i \tilde{C}^i + (\tilde{C}^i)^\top {F^i}^\top \bar{B}^i \prec 0, 
    \label{eq:BMI_2}
\end{align}
where $\bar{B}^i =  \begin{bmatrix} 0_{m_{i} \times n} &  {B^i}^\top P^i &0_{m_{i} \times s_{i}} &I_{m_{i}} \end{bmatrix}$. From Lemma \ref{lem:Finsler}, the above BMI  \eqref{eq:BMI_2} (in $(F^i,P^i)$) is feasible if and only if the following LMIs (in $P^i$) in \eqref{eq:projection_lmi1} are feasible. 
\begin{align}
		\mathcal N_{\tilde{C}^i}^\top \Phi_1^i(P^i)  \mathcal N_{\tilde{C}^i} \prec 0, ~ {\mathcal {N}}_{\bar{B}^{i}}^\top \Phi_1^i (P^i)   \mathcal {N}_{\bar{B}^{i}} \prec 0. 		\label{eq:projection_lmi1}
	\end{align}
Here, $\mathcal{N}_{\bar{B}^i} := \ker(\bar{B}^i)$. The nullspaces $ {\mathcal{N}}_{\bar{B}^{i}}$ and  $\mathcal {N}_{\tilde{B}^{i}}$ are related as $ {\mathcal {N}}_{\bar{B}^{i}} = (I_n \oplus {P^i}^{-1} \oplus I_{s_{i}} \oplus I_{m_i})\mathcal {N}_{\tilde{B}^{i}}$, and using this, \eqref{eq:projection_lmi1} (and as a result, \eqref{eq:LMI1_OF}) is equivalently written as 
\begin{align}
\mathcal  N_{\tilde{C}_{i}}^\top \Phi_1^i \left(P^i \right) \mathcal N_{\tilde{C}^{i}} \prec 0, ~ \mathcal  N_{\tilde{B}^{i}}^\top \Phi_2^i ({P^i}^{-1}) \mathcal N_{\tilde{B}^{i}} \prec 0.  
\label{eq:projection_lmi2} 
\end{align}
Using the Schur complement, the inequality \eqref{eq:LMI2_OF} can be equivalently written as
\begin{align}
\begin{bmatrix}
\delta^i & x_0^\top \\
x_0 & (P^i)^{-1}
\end{bmatrix}
\succ 0.
\label{eq:LMI_ub2_OF}
\end{align}	 
Since the matrix inequalities \eqref{eq:projection_lmi2} and \eqref{eq:LMI_ub2_OF} jointly characterize the set $\mathscr{P}^i$, as defined in \eqref{eq:setPi}. Which implies that set $\mathscr{P}^i \neq \emptyset$. 

\textbf{$\Leftarrow $} Next, we prove the converse: if $\mathscr{P}^i \neq \emptyset$, then there exists an 
$F^i \in \mathfrak{g}^i (\delta^i, F^{-i} )$ such that 
$(A^i + B^i F^i C^i)$ is Schur stable, and  
$J^i(F^i, F^{-i}) < \delta^i$.  To show this, we assume that $\mathscr{P}^i \neq \emptyset$. This implies that there exists 
a matrix $P^i$ satisfying \eqref{eq:projection_lmi2} 
(equivalently, \eqref{eq:projection_lmi1}) and \eqref{eq:LMI_ub2_OF}. 
Using Lemma \ref{lem:Finsler} to \eqref{eq:projection_lmi2}, which implies that there exists $F^i$ such that \eqref{eq:BMI_2} holds. 
After simplifying \eqref{eq:projection_lmi2}, and applying the Schur complement backwardly, \eqref{eq:projection_lmi2} can be written as \eqref{eq:BMI_2} (equivalently, \eqref{eq:LMI1_OF}). 
In a similar manner, applying the Schur complement to \eqref{eq:LMI_ub2_OF} results in \eqref{eq:LMI2_OF}. By assumption, $P^i \in \mathscr{P}^i$ which is a positive-definite matrix. Now, from \eqref{eq:LMI1_OF} and \eqref{eq:LMI_ub2_OF} (equivalently \eqref{eq:LMI2_OF}), and using Lemma \ref{lem:GCi_output_iff}, we conclude that $(A^i +B^iF^iC^i) \subset \mathbb{D}$ and $J^i(F^i, F^{-i}) < \delta^i$. Moreover, for a given $P^i \in \mathscr{P}^i$ such $F^i \in \mathbb{R}^{m_i \times s_{i}}$ exists and obtained using \eqref{eq:FGCi}. 
\end{proof}
\begin{remark}
Theorems \ref{thm:GCEtest_output_iff} and \ref{thm:GCi_OF} can be viewed as discrete-time counterparts of \cite[Theorems 5.4 and 5.7]{Roy:2025guaranteed}, which were established for continuous-time LQ games. The main difference is that, in the continuous-time setting, only sufficient conditions for the existence of an OF-GCE were derived, whereas the present work provides both necessary and sufficient conditions for the existence of an OF-GCE in discrete-time LQ games. 
\end{remark}
\subsection{Algorithm for the synthesis of an OF-GCE}
\label{sec:algorithm_OF}
In the output feedback case, the set $\mathscr{P}^i$ in \eqref{eq:setPi} is non-convex due to the bilinear constraint $P^i W^i = I_n$, where $W^i := (P^i)^{-1}$. As a result, the non-emptiness of $\mathscr{P}^i$ cannot be verified as a convex optimization problem. In what follows, we provide a convex approximation of the set $\mathscr{P}^i$ using a semi-definite programming (SDP) relaxation to verify the non-emptiness of the guaranteed cost response. For computational purposes, it is preferable to work with closed sets. However, the sets $\mathscr{X}^i$  and $\mathscr{Y}^i$, defined using the LMIs \eqref{eq:XLMI} and \eqref{eq:YLMI}, are open and convex. Therefore, we introduce the closed ``$\epsilon$-approximations'' of the sets $\mathscr{X}^i$ and $\mathscr{Y}^i$. Now, we have $\mathscr X^i_\epsilon := 
\{ X \in \mathbb R^{n\times n} ~\mid~
X \succ 0,~
\mathcal N_{\tilde{C}^i}^\top~
\Phi_1^i(X)~
\mathcal N_{\tilde{C}^i}
\preceq -\epsilon I_{2n+s_i+m_i}
\}$, 
$\mathscr Y^i_\epsilon := 
 \{ Y \in \mathbb R^{n\times n} ~\mid~
Y \succ 0,~
\left[\begin{smallmatrix}
\delta^i & x_0^\top \\
x_0 & Y
\end{smallmatrix} \right]
\succeq \epsilon I_{n+1}, ~\mathcal N_{\tilde{B}^i}^\top~
\Phi_2^i (Y)~
\mathcal N_{\tilde{B}^i}
\preceq -\epsilon I_{2n+s_i+m_i}
\}$, where $\epsilon > 0$. As a result, the sets $\mathscr{X}^i_{\epsilon}$  and $\mathscr{Y}^i_{\epsilon}$ are closed and convex. In particular, these sets are  positive definite cones. Moreover, they satisfy $\mathscr{X}^i_\epsilon \subseteq \mathscr{X}^i$ and $\mathscr{Y}^i_\epsilon \subseteq \mathscr{Y}^i$, and converge to $\mathscr{X}^i$ and $\mathscr{Y}^i$ as $\epsilon \to 0$. Using sets $\mathscr{X}^i_{\epsilon}$ and  $\mathscr{Y}^i_{\epsilon}$, an $\epsilon$-approximation of the non-convex set \eqref{eq:setPi} is given by 
\begin{align}
\mathscr{P}^i_\epsilon
&:= \Big\{P^i, W^i \in \mathbb{R}^{n\times n}~\mid~
P^i \succ 0,\;
W^i \succ 0,\;
P^iW^i = I_n,
\nonumber\\
&\qquad \qquad \qquad \qquad\qquad
P^i \in \mathscr{X}_\epsilon^i,\;
W^i \in \mathscr{Y}_\epsilon^i
\Big\}.
\label{eq:setPi_bilinear}
\end{align}
Note that the feasibility of \eqref{eq:setPi_bilinear}, that is, $\mathscr{P}^i_\epsilon \neq \emptyset$, is a cone complementarity problem \cite{Ghaoui:97}, which can be viewed as an extension of linear complementarity problems \cite{Cottle:09} to the cone of positive definite matrices. By replacing the coupling constraint $P^i W^i = I_n$ using the SDP relaxation \cite{Leibfritz:01}, as $P^i\succeq (W^i)^{-1} \Leftrightarrow \left[\begin{smallmatrix}P^i& I_n \\I_n & W^i\end{smallmatrix}\right]\succeq 0$ in \eqref{eq:setPi_bilinear}, we obtain a convex approximation of the set $\mathscr{P}^i_\epsilon$, which is denoted by $\bar{\mathscr{P}}^i_\epsilon$. Note that $(P^i, W^i) \in \partial \bar{\mathscr{P}}^i_\epsilon$ implies $P^iW^i = I_n$, from which it follows that $\mathscr{P}^i_\epsilon = \partial \bar{\mathscr{P}}^i_\epsilon$. Thus, the feasibility of $\mathscr{P}^i \neq \emptyset$ (as a result the non-emptiness of $\mathfrak{g}^i(\delta^i, F^{-i})$) can be verified by determining whether $\partial \bar{\mathscr{P}}^i_\epsilon \neq \emptyset$. To determine the feasibility of $\mathfrak{g}^i(\delta^i, F^{-i}) \neq \emptyset$ for a given $(\delta^i, F^{-i})$, we follow the procedure adopted in \cite[Algorithm~2]{Roy:2025guaranteed}, which is based on the Sequential Linear Programming Matrix Method (SLPMM); see also \cite{Leibfritz:01}.  The steps of the SLPMM algorithm are similar to \cite[Algorithm~2]{Roy:2025guaranteed}. The detailed steps of the SLPMM algorithm are omitted for brevity. Next, we present an iterative algorithm for computing an OF-GCE that satisfies the conditions \eqref{eq:GCEtest_outputfeedback} of Theorem \ref{thm:GCEtest_output_iff}. Algorithm \ref{alg:1_GCE_OF}, referred to as the sequential guaranteed cost response algorithm, is inspired by the sequential best response algorithm which is widely used in the game theory literature; see, for example, \cite{Tembine:2011book, Roy:22, Roy:2025guaranteed}. The algorithm is initialized with a stabilizing initial guess. For a fixed $F^i$, Step 2 verifies $\bar{\mathscr{P}}^i_{\epsilon} \neq \emptyset$ for each player $i$. The player index is then updated in Step 11, and the process is repeated until an OF-GCE is obtained. Prior to each update, Step 1 verifies whether the current feedback strategies satisfy the conditions in \eqref{eq:GCEtest_outputfeedback}. The algorithm does not stop immediately if a player fails to obtain a feasible feedback strategy in Step 2. Instead, the algorithm stops only when all players consecutively fail to update their strategies. It is possible that the algorithm may not terminate and can cycle forever, with  repeated updates of a subset of agents' strategies. 
{\begin{algorithm}
		\caption{Sequential guaranteed cost response for the synthesis of an OF-GCE} 
		\label{alg:1_GCE_OF}  
		\DontPrintSemicolon  
		\KwData{ given   $(\delta^i,\delta^{-i})$ and initial stabilizing guess $(F^i,F^{-i})$;} 
		\KwIn{$i=1$; \tcp*[r]{\small{player index}} }		 
		\KwIn{$j=0$; \tcp*[r]{\small{number of players who fail to satisfy Theorem \ref{thm:GCi_OF}  consecutively}}}
		\While{$(F^i,F^{-i})\notin \mathscr F^\circ$ (verified using Theorem \ref{thm:GCEtest_output_iff})}
		{
			\eIf {$\mathscr P^i_{\epsilon} \neq \emptyset$ exists satisfying Theorem \ref{thm:GCi_OF} given $F^{-i}$ \tcp*[r]{using \cite[Algorithm 2]{Roy:2025guaranteed}}}
			{ Update $F^i$ using \eqref{eq:FGCi} for a feasible $P^i\in \mathscr P^i_{\epsilon}$;\;
				$j=0$;
			}
			{$j=j+1$;\; 
				\If{$j\geq N$}{\KwResult{Stop, cannot find an OF-GCE;}}}
			$i=(i~\text{mod}~N)+1$; \tcp*[r]{\small{update player index}} }
		\KwResult{Found an OF-GCE;}
\end{algorithm}}
\begin{remark}  
Note that the design of an OF-GCE controllers is centralized but their implementation is distributed. The synthesis of OF-GCE stratgeies can be made independent of global initial state $x_0$ by assuming that $x_0\in  \{z\in \mathbb R^n~|~||z||_2 < r\}$, where $r>0$ is a scaling factor. The
conditions \eqref{eq:GCEoutput2} and \eqref{eq:LMI2_OF} to be replaced equivalently by the LMI in the variable $P^i \succ 0$ as $P^i \prec \frac{\delta^i}{r^2} I_n$, which eliminates the dependence on a specific initial condition; see \cite{Jiao:20TAC}. As a result, the LMI set \eqref{eq:YLMI} can be equivalently written as for player $i \in \mathsf{N}$
\begin{align}
    \mathscr Y^i
:=
\bigl \{
Y \in \mathbb R^{n\times n}
~\mid~
Y \succ 0,
Y - \tfrac{r^2}{\delta^i} I_n \succ 0,\;
\mathcal N_{\tilde{B}}^\top \Phi_2^i(Y)\,\mathcal N_{\bar{B}} \prec 0
\bigr \}. 
\end{align}
\label{rem:ind_init_cond}
\end{remark}
\begin{remark}
\label{rmk:init_stableSoln} Algorithm~\ref{alg:1_GCE_OF} require that the initial stabilizing feedback gains $(F^i, F^{-i})$. For a given $P \succ 0$, the feasibility of 
 $(A + \sum_{i=1}^N B^i F^i C^i )^\top P (A + \sum_{i=1}^N B^i F^i C^i ) - P \prec 0$, with respect to the variables $(F^{i}, F^{-i})$ provides an initial set of stabilizing feedback gains. Note that, for a given $P \succ 0$, the above matrix inequality is not linear in the variables $(F^i, F^{-i})$. Hence,  using the Schur complement, we obtain the following LMI condition 
 \begin{align}
     \begin{bmatrix}
         -P & (A + \sum_{i=1}^N B^i F^i C^i )^\top P \\
    P (A + \sum_{i=1}^N B^i F^i C^i ) & -P
     \end{bmatrix}
     \prec 0.
 \end{align}
\end{remark}
\begin{remark}
Algorithm \ref{alg:1_GCE_OF} depends on sequential guaranteed cost responses. 
It is closely related to best-response dynamics used to compute Nash equilibria. It is well known that best-response dynamics converge only for certain classes of games; see \cite{Tembine:2011book}. Similarly, no theoretical convergence guarantees for Algorithm \ref{alg:1_GCE_OF} are established within the scope of this note. Moreover, the output feedback case leads to a nonconvex problem, which introduces additional complexity. As discussed earlier, this motivates the use of the SLPMM algorithm in Step 2 of Algorithm \ref{alg:1_GCE_OF} to verify that $\mathscr{P}^i_\epsilon \neq \emptyset$. This difficulty is inherent to the design of static output feedback suboptimal controllers and has been discussed in the literature; see also discussions in \cite{Leibfritz:01, Iwasaki:94}. We also note that the selection of initial stabilizing feedback gains can significantly influence the convergence rate and performance of Algorithm \ref{alg:1_GCE_OF}. 
\end{remark}
\begin{remark} We note that the selection of the cost profile $\left(\delta^i, \delta^{-i} \right) \in \mathbb{R}^N_+$ plays a crucial role in the existence of an OF-GCE. In particular, the non-emptiness of a guaranteed cost of each player is not only depends on her own cost estimate $\delta^i$ but also on the cost estimates of the other players, $\delta^{-i}$; see Theorem \ref{thm:GCi_OF}. Using the monotonicity property of OF-GCE (see Theorem \ref{thm:LargeDelta}), a practical approach in Algorithm \ref{alg:1_GCE_OF} is to initialize the cost profile with relatively large values and then gradually decrease them until no OF-GCE can be found. 
\end{remark}
\section{Existence and Synthesis of State feedback guaranteed cost equilibrium}
\label{sec:SF_GCE}
In this section, we consider the case of complete state information, where $C^i = I_n,~\forall i \in \mathsf{N}$. Thus, Theorem~\ref{thm:GCEtest_output_iff} provides necessary and sufficient conditions for a strategy profile to constitute a stabilizing SF-GCE. However, synthesizing an SF-GCE requires that the guaranteed cost response set \eqref{eq:satisfactionfunction} be non-empty. The following result provides a necessary and sufficient condition for this set to be non-empty.
\begin{corollary}
\label{thm:GCi_SF}
Let $\delta^i>0$ and $F^{-i} \in \bigtimes_{j \in -i} \mathbb R^{m_j \times n} $ be given.    Consider the following set
	\begin{subequations} 
		\begin{align}
			\resizebox{0.98\columnwidth}{!}{$
\mathscr{Y}^i
:=
\Big\{
Y \in \mathbb{R}^{n\times n}
~\big|~
Y \succ 0,\;
Y - \frac{r^2}{\delta^i} I_n \succ 0,\;
\mathcal{N}_{\hat{B}^i}^\top
\Phi^i(Y)
\mathcal{N}_{\hat{B}^i}
\prec 0
\Big\}
$}
\label{eq:YLMI_SF}
           \end{align} 
        \end{subequations}
    \text{where,}
\begin{subequations}
    \begin{align}
			\Phi^i (Y)&:=\begin{bmatrix} 
                -Y  &Y {A^i}^\top &\left (\sqrt{Q^i}Y \right)^\top  & 0_{n\times  {m}_i }\\
                A^iY  &-Y &0_{n \times n} &0_{n \times m_{i}} \\
				\sqrt{Q^i}Y   &0_{n \times n} & -I_{n} &0_{ n \times m_i}\\
				0_{{m}_i\times n} &0_{{m}_i\times n} &0_{m_i \times n}&-(R^i)^{-1}
				 \end{bmatrix},
        \end{align}
        \end{subequations}
		and $A^i:=A+\sum_{j\in -i}B^j F^j$. The matrix $\mathcal N_{\hat{B}^i}:= \ker(\hat{B}^i)$ denote matrix with orthonormal columns which span the null space of the matrix $\hat{B}^i := \begin{bmatrix} 0_{m_{i} \times n} &{B_i}^\top & 0_{m_{i} \times n} &I_{m_i}\end{bmatrix}$. Define the set
		\begin{align}
			\mathscr P^i&:=\{P^i\in \mathbb R^{n\times n} ~|~P^i\succ 0,  ~(P^i)^{-1} \in \mathscr Y^i \}.\label{eq:setPi_SF}
		\end{align}
		Then, there exists an $ F^i \in \mathfrak{g}^i(\delta^i, F^{-i}) \neq \emptyset$, and  for any $F^i \in \mathfrak{g}^i(\delta^i, F^{-i}) $ we have $\sigma(A^i + B^i F^i C^i) \subset \mathbb{D}$ if and only if $ \mathscr{P}^i \neq \emptyset $. In this case, any $ F^i \in \mathfrak{g}^i(\delta^i, F^{-i}) $ satisfies the following LMI for some $ P^i \in \mathscr{P}^i $ 
	\begin{align}
\resizebox{0.98\columnwidth}{!}{$
\begin{bmatrix}
-P^i
&
\bigl(A^i + B^iF^i\bigr)^\top P^i
&
\bigl(\sqrt{Q^i}\bigr)^\top
&
\bigl(\sqrt{R^i}F^i\bigr)^\top
\\
P^i\bigl(A^i + B^iF^i\bigr)
&
-P^i
&
0_{n\times n}
&
0_{n\times m_i}
\\
\sqrt{Q^i}
&
0_{n\times n}
&
-I_n
&
0_{n\times m_i}
\\
\sqrt{R^i}F^i
&
0_{m_i\times n}
&
0_{m_i\times n}
&
-I_{m_i}
\end{bmatrix}
\prec 0.
$}
\label{eq:FGi_SF}
\end{align} 
\label{cor:FGi_SF}
\end{corollary}
\begin{proof}
The proof uses the result from Remark~\ref{rem:ind_init_cond} and follows the same steps as the proof of Theorem~\ref{thm:GCi_OF}.
\end{proof}
\begin{remark}
In \cite[Theorem 5.3]{Roy:2025guaranteed}, we established a sufficient condition for synthesizing SF-GCE strategies in continuous-time game using a change-of-variables approach. In contrast, in this work, we derive necessary and sufficient conditions in Corollary \ref{cor:FGi_SF} for the synthesis of SF-GCE strategies. Moreover, unlike previous approaches, the proposed state feedback result does not rely on a change-of-variables technique. Instead, a projection-lemma-based approach unifies both the state and output feedback cases, with the state feedback case arising as a special case of the output feedback case. 

\end{remark}
\begin{remark}
In the output feedback case, we observed that the set $\mathscr P^i$ in \eqref{eq:setPi} is non-convex. However, in the state feedback case, this set \eqref{eq:setPi_SF} is convex as the set is characterized by the convex set $\mathscr{Y}^i$ in \eqref{eq:YLMI_SF}. As a result, the SF-GCE strategies can be synthesized using LMIs without the need for SDP relaxation techniques used in Section~\ref{sec:algorithm_OF}. Note that, for computational purposes, we prefer closed sets. Hence, we consider the closed ``$\epsilon$-approximation" of the convex set \eqref{eq:YLMI_SF}. Finally, we use the sequential guaranteed cost response approach as presented in Algorithm~\ref{alg:1_GCE_OF} to compute an SF-GCE. We note that Algorithm~\ref{alg:1_GCE_OF} starts with initial stabilizing feedback gains obtained using the procedure discussed in Remark~\ref{rem:ind_init_cond}, where $ C^i = I_n, \forall i \in \mathsf{N}$. 
\end{remark}
\begin{remark}
As discussed in Section \ref{sec:preliminaries}, computing a feedback Nash equilibrium for LQ differential games involves solving a set of coupled algebraic equations that are quadratic in the decision variables, whereas for LQ difference games, the corresponding equations are non-quadratic. By contrast, computing a feedback guaranteed equilibrium (in both differential and difference game settings) is related to a set of BMIs. However, these can be reformulated as LMIs in the decision variables (see Theorem \ref{thm:GCi_OF} and Corollary \ref{cor:FGi_SF}), which makes the problem computationally tractable. 
\end{remark}
\section{Numerical illustrations}
\label{sec:numerical_results}
In this section, we demonstrate the performance of SF-GCE and OF-GCE through numerical examples. 
\begin{example}
\label{ex:scalar-2ply}
We consider a $2$-player infinite-horizon LQ difference game with complete state observation. The state dynamics is given by $x_{k+1}  = Ax_k + \sum_{i=1}^2 B^i u^i_k, x_0 \in \mathbb{R}$ where $x_k \in \mathbb{R}$ and $u^i_k \in \mathbb{R}$. Each player $i$ has the cost functional $J^i = \sum_{k=0}^{\infty} (Q^i x_k^2 + R^i {u_k^i}^2)$, $i=1,2$. The game parameters are $A =2.1$, $B^1 = 2$, $B^2=1$, $Q^1 = 0.45$, $Q^2 = 0.25$, $R^1 = 5$, and $R^2 = 0.65$, with initial condition $x_0 = 0.35$. For this game, we obtain three stabilizing SF-NE. The corresponding SF-NE costs are presented in Table \ref{tab:NE_PoS}. We obtain the cooperative cost by solving \eqref{eq:tgcostcoop} as $J_{\mathrm{Co}} = 0.2804$. To determine the SF-GCE region, we consider the parameter region $(\delta^1, \delta^2) \in [0, 0.5] \times [0, 0.5]$ and uniformly  sample $350$ points within this region. For each sampled $(\delta^1, \delta^2)$, we find out the existence of an SF-GCE using Algorithm~\ref{alg:1_GCE_OF}. The shaded region in Fig. \ref{fig:gce_existence} indicates the parameter values for which SF-GCE exist. Furthermore, the shaded regions in Fig. \ref{fig:cost_comp} show the SF-GCE costs of the players. Note that the SF-GCE cost region under non-cooperative behavior lies above the Pareto frontier, that is, the costs achievable by the players in cooperation. Clearly, there exists an SF-GCE that can support a cooperative outcome. Moreover, Fig.~\ref{fig:cost_comp} shows that the SF-NE costs associated with the three SF-NE lie in the interior of the cost set generated by SF-GCE, which verifies Theorem~\ref{thm:GCE_Nash}. Fig.~\ref{fig:pos_region} shows that the PoS associated with SF-GCE satisfies the upper bound \eqref{eq:PoS1}. Fig.~\ref{fig:pos_magnified_NearNash} is a magnified version of Fig.~\ref{fig:pos_region}, which shows that the PoS values correspond to each SF-NE; see also Table \ref{tab:NE_PoS}. From Fig.~\ref{fig:pos_magnified_NearNash}, we observe that there are many SF-GCE at which PoS values are close to $1$, while several others achieve PoS values strictly lower than the best PoS value of $1.1254$ at an SF-NE; see also Fig.~\ref{fig:pos_magnified_Near1}, the magnified version of Fig.~\ref{fig:pos_magnified_NearNash}. The minimum PoS obtained for this problem is $1.0002$. 
\begin{table}[t]
\centering
\caption{PoS at each SF-NE in Example \ref{ex:scalar-2ply}}
\label{tab:NE_PoS}
\begin{tabular}{c|c|c|c}
  & SF-NE &Cost at SF-NE &PoS \\ [.5ex]
\hline
$\mathrm{NE}_1$ &$(0.6859, 1.8901) $ 
& $(0.0840, 0.2315)$ 
& $ 1.1254$ \\

$\mathrm{NE}_2$ &$(3.4306, 0.3998)$ 
& $(0.4202, 0.0490)$ 
& $1.6734$ \\

$\mathrm{NE}_3$ &$(1.8463, 0.8603)$ 
& $(0.2262, 0.1054)$ 
& $1.1824$ \\
\hline
\end{tabular}
\end{table}

\begin{figure}
    \centering
    \begin{subfigure}{0.48\linewidth}
        \centering
  \includegraphics[width=\linewidth, height = 0.825\linewidth]{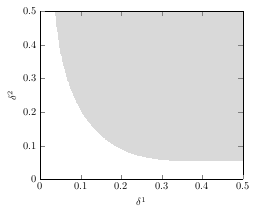}
        \caption{}
        \label{fig:gce_existence}
    \end{subfigure}
    \hspace{1pt}
    \begin{subfigure}{0.48\linewidth}
        \centering
     \includegraphics[width=\linewidth, height = 0.825\linewidth]{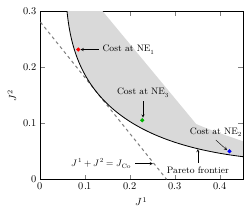}
        \caption{}
        \label{fig:cost_comp}
    \end{subfigure}
   \vspace{8pt}
    \begin{subfigure}{0.48\linewidth}
        \centering
        \includegraphics[width=\linewidth, height = 0.825\linewidth]{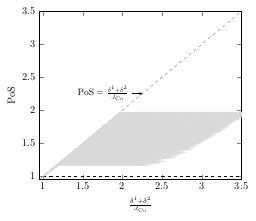}
        \caption{}
        \label{fig:pos_region}
    \end{subfigure}
    \hspace{1pt}
    \begin{subfigure}{0.48\linewidth}
        \centering
        \includegraphics[width=\linewidth, height = 0.825\linewidth]{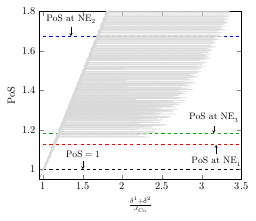}
        \caption{}
        \label{fig:pos_magnified_NearNash}
    \end{subfigure}
    \vspace{8pt}
    \begin{subfigure}{0.48\linewidth}
        \centering
        \includegraphics[width=\linewidth, height = 0.825\linewidth]{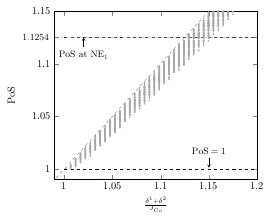}
        \caption{}
        \label{fig:pos_magnified_Near1}
    \end{subfigure}
\caption{Results for Example~\ref{ex:scalar-2ply}: SF-GCE existence region (panel (a)), SF-NE costs in the interior of the SF-GCE cost region (panel (b)), PoS at SF-GCE region (panel (c)), magnified view of panel (c) (panel (d)), and magnified view of panel (d) where PoS $\approx 1$ (panel (e)).}
\end{figure}
\end{example}
\begin{example}
\label{ex:5ply_heteregeneous}
 We consider an output-consensus problem for a heterogeneous $5$-agent system that communicates over a directed graph; see Fig.~\ref{fig:5agentnetwork_OF_GCE}.
 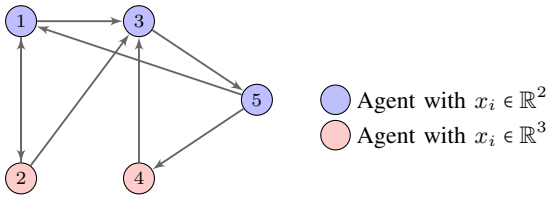
\begin{figure}[H] 
    \centering 
    \begin{tikzpicture}[scale=.26,>=latex', inner sep=1mm, font=\small]
        \tikzstyle{solid node1}=[circle,auto=center,draw,minimum size=7pt,inner sep=2,fill=blue!25]
        \tikzstyle{solid node2}=[circle,auto=center,draw,minimum size=7pt,inner sep=2,fill=red!20]
        \tikzstyle{dedge} = [draw, ->, black!60, line width=.25mm] 
        \def \x{2.5};
        \def \y{4};
        
        \node (n1)[solid node1] at (6,\y) {\scriptsize{$1$}};
        \node (n2)[solid node2] at (6,-\y) {\scriptsize{$2$}}; 
        \node (n3)[solid node1] at (12,\y) {\scriptsize{$3$}};
        \node (n4)[solid node2] at (12,-\y) {\scriptsize{$4$}};
        \node (n5)[solid node1] at (18,0) {\scriptsize{$5$}};   
        \path[dedge] (n1) -- (n2);
        \path[dedge] (n2) -- (n1);	 
        \path[dedge] (n1) -- (n3);	 
        \path[dedge] (n5) -- (n4); 	 
        \path[dedge] (n3) -- (n5);	 
        \path[dedge] (n4) -- (n3);	
        \path[dedge] (n2) -- (n3);	
        \path[dedge] (n5) -- (n1);	

        \draw[fill=blue!25](22,0) circle[radius=0.75];
        \draw[fill=red!20](22,-1.8) circle[radius=0.75];
        \draw(22,0)node[right,xshift=5]{Agent with $x_i \in \mathbb{R}^2$};
        \draw(22,-1.8)node[right,xshift=5]{Agent with $x_i \in \mathbb{R}^3$};
    \end{tikzpicture}
    \caption{Directed communication graph of the 5-agent heterogeneous multi-agent system in Example~\ref{ex:5ply_heteregeneous}.}
    \label{fig:5agentnetwork_OF_GCE}
\end{figure}
This problem has been studied in the continuous-time game framework in \cite[Example 7.3]{Roy:2025guaranteed}. The agent dynamics are discretized with a sampling time $ T_s = 0.1$. The local information, objective functional of each agent, remains the same as in \cite[Example 7.3]{Roy:2025guaranteed}. Consequently, the weighting matrices $Q^i$ and $R^i$, $i \in \mathsf{N}$, are unchanged. This example illustrates the performance of an OF-GCE in discrete-time games with heterogeneous, networked agents. For simulation purpose, we choose $x_0=[-0.3 ~0.5 ~0.4 ~0.2 ~0.6 ~-0.3 ~0.2 ~-0.1 ~0.5 ~0.7 ~0.2 ~-0.4]^\top$. By following the procedure discussed in Remark \ref{rmk:init_stableSoln}, we consider the following set of initial stabilizing feedback strategies. 
\begin{align*}
		&F^1 = 
		\left[\begin{smallmatrix}
			 15.6163  &-10.8561  & -1.0446   &-0.5333   & 6.3086   &-6.3297 \\
  -12.9511    &5.3031    &0.6990    &1.3441   &-6.2039    &5.9953
		\end{smallmatrix} \right],  
		\\
       &F^2= 
		\left[\begin{smallmatrix}
			 0.9515   &-0.6758   &-1.4371   &-0.3797 \\
             -4.6275   &3.6151   &-0.7442   &-7.2380
		\end{smallmatrix} \right], 
		\\
		&F^3 = 
		\left[\begin{smallmatrix}
			 0.8763  & -0.4784   &-0.0868    &0.8174   &13.2609   &-9.8851   &-0.7422    &0.4551 \\
   -0.2105   &-0.1562    &0.2125   &-1.4365   &-6.8430    &1.6554    &0.7042   &-1.5720
		\end{smallmatrix} \right], \\
		&F^4 = 
		\left[\begin{smallmatrix}
			-1.6056   &-0.1640   &-0.1783    &0.1198 \\
    0.5590   &-3.2441    &5.7167   &-5.0583
		\end{smallmatrix} \right],  \\
		&F^5  = 
		\left[\begin{smallmatrix}
			 -2.4364  &2.5555   &16.1967  &-11.1420 \\
    2.9683   &-2.0417  &-15.9418  &8.7156
		\end{smallmatrix} \right].
	\end{align*}	
	Using Algorithm \ref{alg:1_GCE_OF}, we obtain the following set of OF-GCE strategies. 
	\begin{align*}
		&F^{1\,\circ} = 
		\left[\begin{smallmatrix}
		3.3701   &-3.3138   &-0.0230   &0.1387   &-0.3159    &0.4166 \\
      -2.6807   &-0.0558    &0.0167   &-0.0530   &-0.2933    &0.2913
		\end{smallmatrix} \right], \\
		&F^{2\,\circ} = 
		\left[\begin{smallmatrix}
			 0.9602   &-0.4829   &-1.4835   &-0.1372 \\
              -0.5816  &0.6657    &0.2650   &-1.4032
		\end{smallmatrix} \right], \\
		&F^{3\,\circ} = 
		 \left[\begin{smallmatrix}
			-0.0370    &0.1792   &-0.0468    &0.4188    &3.1238   &-3.1030   &-0.0515    &0.3128 \\
    0.1542   &-0.0617    &0.0261    &0.1973   &-2.7254   &-0.0962    &0.0182    &0.0914
		\end{smallmatrix} \right], \\
		&F^{4\,\circ} = 
		\left[\begin{smallmatrix}
		-1.3222   &-0.0871    &0.4642   &-0.1681 \\
    0.1180   &-1.1506   &-0.6568    &0.8678 
		\end{smallmatrix} \right], \\
		 &F^{5\,\circ}= 
		\left[\begin{smallmatrix}
			0.3079    &0.0538    &2.2063   &-2.1547 \\
                0.4415    &-0.2566   &-3.7581   &0.6579
		\end{smallmatrix} \right].
	\end{align*}
    In Table \ref{tab:OFGCE_5agent}, we demonstrate that at an OF-GCE, the costs of all the agents lie below than a given cost profile. Next, we compute the total game cost in cooperation using \eqref{eq:tgcostcoop}, and obtain $ J_{\mathrm{Co}} = 21.7374$. We then calculate the associated PoS for the OF-GCE. We obtain $\mathrm{PoS} = \frac{\sum_{i=1}^5 J^{i\,\circ}}{J_{\mathrm{Co}}} = 1.2181 < \frac{\sum_{i=1}^{5} \delta^i }{J_{\mathrm{Co}}} = 2.6912$, which verifies Theorem \ref{thm:PoS}.  In Figure \ref{fig:Case1magn_traj_OF}, we observe that the agents' error trajectories converge asymptotically to zero.
    \begin{table}
	\centering  
	\caption{Individual costs and PoS for the OF-GCE strategies for Example~\ref{ex:5ply_heteregeneous}}
	\begin{tabular}{ p{0.30cm}|p{0.65cm}|p{0.65cm}|p{0.65cm}|p{0.65cm}|p{0.65cm}|p{1.10cm}|p{0.65cm}}
		 & $1$ &$2$ &$3$ &$4$ &$5$ &total cost &$\mathrm{PoS}$ \\ [.5ex]
       \hline  
		$\delta^i$   &$13.5$  &$9$  &$13.5$ &$9$  &$13.5$   &$58.5$  &$2.6912$ \\ 
        $J^{i\,\circ}$  &$9.6630$  &$1.8385$   &$ 3.5883$ &$5.3421$  &$6.0485$ &$26.4803$ &$1.2182$ \\ [.5ex]
        \hline 
	\end{tabular}
	\label{tab:OFGCE_5agent}
\end{table}
    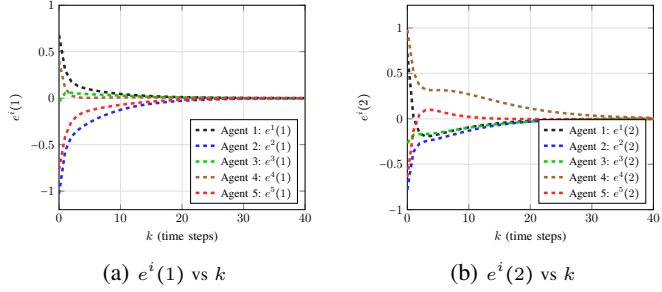
\begin{figure}
    \centering 
		\subfloat[$e^i(1)$ vs $k$]
		{\input{FigData/ei_1_OFGCE_ex4} \label{fig:ei_1_OFGCE_ex4}}~~
		\subfloat[$e^i(2)$ vs $k$]{\input{FigData/ei_2_OFGCE_ex4}
			\label{fig:ei_2_OFGCE_ex4}}     
		\caption{Panel (a) and Panel (b) illustrate the error trajectories of the agents using OF-GCE strategies for Example~\ref{ex:5ply_heteregeneous}} 
		\label{fig:Case1magn_traj_OF}
	\end{figure}
\end{example}
\section{Conclusions}
\label{sec:concl}
In this note, we studied the solution concept output feedback guaranteed cost equilibrium (OF-GCE) for infinite-horizon linear-quadratic deterministic difference games with an output feedback information structure. At an OFGCE, each player's cost is upper-bounded by a given threshold. We provide necessary and sufficient conditions for the existence of these equilibria. Additionally, we also present a linear matrix inequality-based iterative algorithm for synthesizing them. We have also shown that state feedback GCE (SF-GCE) can be obtained as a special case of an OF-GCE when players have the complete state information. Future work will focus on establishing theoretical convergence guarantees for the proposed iterative algorithm. Further, we will analyze the GCE concept for discrete-time scalar LQ games to gain analytical insight and extend the GCE framework to stochastic dynamic games. 
\vspace{-2pt}
\bibliographystyle{IEEEtran}
\bibliography{GCEreference}

\end{document}

%% file: FigData/ei_1_OFGCE_ex4.tex
\begin{tikzpicture}[scale=.475,>=latex']
	\tikzset{every pin/.append style={font=\small}}
	\begin{axis}[xmin=0,xmax=40,ymin=-1.2, ymax=1, xlabel = {$k$ (time steps)},ylabel={$e^{i}(1)$},legend pos = south east,
		grid=both, legend style={nodes={scale=0.9}}, 
		grid style={line width=.5pt, draw=gray!25},] 
		\addplot[color=black!85, dashed, line width=2pt] table{FigData/e1_1_ex4.dat}; \addlegendentry{Agent 1: $e^{1}(1)$};
		\addplot[color=blue!85, dashed, line width=2pt] table{FigData/e2_1_ex4.dat}; \addlegendentry{Agent 2: $e^{2}(1)$};
		\addplot[color=green!85!black, dashed, line width=2pt] table{FigData/e3_1_ex4.dat}; \addlegendentry{Agent 3: $e^{3}(1)$};
		\addplot[color=brown!85!black, dashed, line width=2pt] table{FigData/e4_1_ex4.dat}; \addlegendentry{Agent 4: $e^{4}(1)$}
		\addplot[color=red!85, dashed, line width=2pt] table{FigData/e5_1_ex4.dat}; \addlegendentry{Agent 5: $e^{5}(1)$}
	\end{axis}  
\end{tikzpicture} 

%% file: FigData/ei_2_OFGCE_ex4.tex
\begin{tikzpicture}[scale=.475,>=latex']
	\tikzset{every pin/.append style={font=\small}}
	\begin{axis}[xmin=0,xmax=40,ymin=-1, ymax=1.25, xlabel = {$k$ (time steps)},ylabel={$e^{i}(2)$},legend pos = south east,
		grid=both, legend style={nodes={scale=0.9}}, 
		grid style={line width=.5pt, draw=gray!25},] 
		\addplot[color=black!85, dashed, line width=2pt] table{FigData/e1_2_ex4.dat}; \addlegendentry{Agent 1: $e^{1}(2)$};
		\addplot[color=blue!85, dashed, line width=2pt] table{FigData/e2_2_ex4.dat}; \addlegendentry{Agent 2: $e^{2}(2)$};
		\addplot[color=green!85!black, dashed, line width=2pt] table{FigData/e3_2_ex4.dat}; \addlegendentry{Agent 3: $e^{3}(2)$};
		\addplot[color=brown!85!black, dashed, line width=2pt] table{FigData/e4_2_ex4.dat}; \addlegendentry{Agent 4: $e^{4}(2)$}
		\addplot[color=red!85, dashed, line width=2pt] table{FigData/e5_2_ex4.dat}; \addlegendentry{Agent 5: $e^{5}(2)$}
	\end{axis}  
\end{tikzpicture} 